\documentclass{siamltex}		
\usepackage{floatrow}
\usepackage{graphicx,multirow,comment,amssymb,amsmath,amsfonts,subfigure,
hyperref}
\usepackage{enumitem,array}%
\usepackage{exscale,algorithm,algpseudocode}
\usepackage{multirow}
\usepackage[normalem]{ulem}
\newcounter{remark}

\numberwithin{remark}{section}
\newcounter{example}

\numberwithin{example}{section}

\def\inv{^{-1}}%

\newcommand{\eq}[1]{\begin{equation}\label{#1}}
\newcommand{\en}{\end{equation}}

\graphicspath{{./FIGS/}}

\title{A power Schur complement low-rank correction preconditioner for
general sparse linear systems\thanks{This work was supported
by NSF under grant NSF/DMS 1912048, Shuimu Scholar of Tsinghua University and
by the Minnesota Supercomputing Institute.}}

\author{
Qingqing Zheng\thanks{Department of Mathematical Sciences,
Tsinghua University, 100084 Beijing, China.
{\tt \{zheng1990@mail.tsinghua.edu.cn\}}}
\and
Yuanzhe Xi
\thanks{Department of Mathematics,
        Emory University.
{\tt \{yxi26@emory.edu\}}}
\and
Yousef Saad\thanks{Computer Science \& Engineering,
        University of Minnesota, Twin Cities, US.
{\tt \{saad@umn.edu\}}}
}

\begin{document}
\maketitle

\begin{abstract}
  A parallel preconditioner is proposed for general large sparse
  linear systems that combines a power series expansion method with
  low-rank correction techniques. To enhance convergence, a power
  series expansion is added to a basic Schur complement iterative
  scheme by exploiting a standard matrix splitting of the Schur complement.
  One of the goals of the power series approach is to improve the
  eigenvalue separation of the preconditioner thus allowing an
  effective application of a low-rank correction
  technique. Experiments indicate that this combination can be quite
  robust when solving highly indefinite linear systems.  The
  preconditioner exploits a domain-decomposition approach and its
  construction starts with the use of a graph partitioner to reorder
  the original coefficient matrix.  In this framework, unknowns
  corresponding to interface variables are obtained by solving a
  linear system whose coefficient matrix is the Schur
  complement. Unknowns associated with the interior variables are
  obtained by solving a block diagonal linear system where parallelism
  can be easily exploited.  Numerical examples are provided to
  illustrate the effectiveness of the proposed preconditioner, with an emphasis on
  highlighting its robustness properties in the indefinite case.
\end{abstract}



\section{Introduction}
Consider the solution of the following linear system
\begin{equation}\label{equ:original}
	Az=b,
\end{equation}
where $A\in \mathbb{R}^{n\times n}$ is a large sparse matrix and
$b\in \mathbb{R}^{n}$ is a given vector. Preconditioned Krylov
subspace methods are often used for solving such systems, see,
e.g., \cite{Saad-book2}.
Among the most popular general-purpose preconditioners are the
Incomplete LU (ILU) techniques \cite{m12,m13}. However, ILU  often
fails, especially in situations when the matrix is highly
indefinite \cite{doi:10.1137/18M1228128,doi:10.1137/16M1078409}. In addition, due to their sequential nature, ILU
preconditioners will result in poor performance on massively parallel
high-performance computers.  Algebraic multigrid (AMG) methods
constitute another class of popular techniques for solving problems
arising from some discretized elliptic PDEs. Often, AMG also fails
for indefinite problems. Finally,
sparse approximate inverse preconditioners
\cite{m2,m4,m3,jia1} were developed to overcome these
shortcomings but were later abandoned by practitioners due to their high
memory demand.

Recently, a new class of approximate inverse preconditioners based on
low-rank approximations has been proposed. They include the Multilevel
Low-Rank (MLR) preconditioner \cite{m5}, the Schur complement low-rank
(SLR) preconditioner \cite{m7}, the Multilevel Schur complement
Low-Rank (MSLR) preconditioner \cite{m8} and the Generalized
Multilevel Schur complement Low-Rank (GMSLR) preconditioner \cite{m9}.
These preconditioners approximate the Schur complement or its inverse
by exploiting various low-rank corrections and because they are
essentially approximate inverse methods they tend to perform rather well on
indefinite linear systems. Similar ideas have also been exploited in \cite{doi:10.1137/16M1109503}. A related class of methods is the class of rank
structured matrix methods, which include the HOLDR-matrix
\cite{holdor}, the $\mathcal{H}$-matrix \cite{bebendorfbook,Borne06},
the $\mathcal{H}^2$-matrix \cite{hackhss2002} and hierarchically
semiseparable (HSS) matrices \cite{smash,doi:10.1137/15M1023774,Xi-Xia}. These methods
partition the coefficient matrix $A$ into several smaller blocks and
approximate certain off-diagonal blocks by low-rank matrices. These
techniques have recently been applied to precondition sparse linear
systems, resulting in some rank structured sparse preconditioners.  We
refer the reader to \cite{supermfs,supermf,selected-inversion,doi:10.1137/17M1124152} for
details.

In this paper, we present a method that combines low-rank approximation
methods with a simple  Neumann polynomial expansion
technique \cite[Section 12.3.1]{Saad-book2} aimed at improving robustness.
We call the resulting method the
\emph{Power -- Schur complement Low-Rank} (PSLR)
preconditioner. A straightforward way to
apply the Neumann polynomial preconditioning technique
to the Schur complement $S$ is to approximate $(\omega S)^{-1}$ by an $m$-term polynomial expansion as \cite[Section 12.3.1]{Saad-book2}
\begin{equation}\label{Napp}
\frac{1}{\omega}\big[I+N+N^{2}+\cdots+N^{m}\big]D^{-1},
\end{equation}
where $\omega$ is a scaling parameter, $D$ is the (block) diagonal of
$S$ and $N=I-\omega D^{-1}S$. However, scheme \eqref{Napp} has a
number of disadvantages.  For example it is difficult to choose an
optimal value for the parameter $\omega$.  In addition, since the matrix
series in (\ref{Napp}) converges only when $\rho(N)<1$, the
approximation accuracy will improve as $m$ increases only under this
condition which may not be satisfied for a general matrix.  Moreover,
even if $\rho(N)<1$, \eqref{Napp} is only a rough approximation to
$S^{-1}$ when $m$ is small and using a large $m$ may become
computationally expensive. The PSLR preconditioner seamlessly combines
the power series expansion with a few low-rank correction techniques
and can overcome these shortcomings.  We summarize below the main
advantages of the PSLR preconditioner over existing low-rank
approximate inverse preconditioners.

\begin{enumerate}
\item \textbf{Improved robustness}. When $\rho(N)> 1$, the classical
  Neumann series defined by \eqref{Napp} diverges and the
  approximation accuracy deteriorates as $m$ increases. However,
  low-rank correction techniques can be invoked to address this
  issue. More specifically, we exploit low-rank correction techniques
  as a form of deflation to move those eigenvalues of $N$ with modulus
  larger than $1$ closer to $0$. The goal is to make the series
  \eqref{Napp} converge for the ``deflated" Schur complement.

\item \textbf{Enhanced decay property}.  The performance of each of the
  three previously developed methods, SLR,
  MSLR and GMSLR,  depends on the eigenvalue decay
  property associated with the Schur complement inverse $S^{-1}$. If
  the decay rate is slow, these preconditioners are not effective.  On
  the other hand, PSLR preconditioner can control the eigenvalue decay
  rate of the matrix to be approximated  by adjusting  the number of the
  expansion term $m$ in \eqref{Napp}
  and this can  significantly improve performance.

\item \textbf{High parallelism}.  The low-rank correction terms used
  in the PSLR preconditioner can be computed by solving several linear
  systems with coefficient matrices that are block diagonal.  This
  results in a much more efficient treatment than with in MSLR and
  GMSLR preconditioners since ILU factorizations and the resulting
  triangular solves can be applied efficiently in parallel.  In
  addition, most of the important  matrix-vector products of  PSLR
  involve  block diagonal matrices or  dense
  matrices, leading to  a high degree of parallelism in both
  the construction and the application stage.
\item \textbf{Suitability for general matrices}.
  PSLR {is quite effective in handling general sparse problems.}
  Unlike  SLR and MSLR, it is not restricted to
 symmetric systems. Numerical experiments in Section \ref{sec:num}
 illustrate that the PSLR preconditioner outperforms
 the other low-rank approximation based preconditioners on various tests.
\end{enumerate}

The paper is organized as follows.  Section \ref{sec:order} is a
brief review of graph partitioning, which will be used to reorder the
original coefficient matrix $A$.  Section \ref{sec:pre} shows how to
build the PSLR preconditioner by exploiting low-rank approximations
and a power series expansion associated with the inverse of a certain
Schur complement $S$.  A spectral analysis {for the corresponding
  preconditioned matrix} is also developed.  Section \ref{sec:num}
reports on numerical experiments to illustrate the efficiency and
robustness of the PSLR preconditioner. Concluding remarks are
stated  in Section \ref{sec:con}.

\section{Background:
graph partitioning}\label{sec:order}
Building the PSLR preconditioner begins with a reordering
of the coefficient matrix $A$ with the help of a
graph partitioner \cite{m7,Saad-book2}. Specifically, in this
paper, we invoke any vertex-based (aka `edge separation') partitioner
to reorder $A$. As there is no ambiguity,
we will still use $A$ and $b$ to denote the reordered
matrix and right-hand side, respectively.

Let $s$ be the number of subdomains used in the partitioning. When
the variables are labeled by subdomains and the interface variables
are labeled last, the permuted linear system of \eqref{equ:original}
can be rewritten as
\begin{equation}\label{equ:resystem}
Az=\begin{pmatrix}
B  &E\\
F&C\\
\end{pmatrix}\begin{pmatrix}
x\\
y\\
\end{pmatrix}=\begin{pmatrix}
f\\
g\\
\end{pmatrix},
\end{equation}
where $B\in\mathbb{R}^{p\times p},E\in\mathbb{R}^{p\times q}$
and $F\in\mathbb{R}^{q\times p}$ with $p+q=n$.
The submatrices $B, E, C$ have the following block diagonal structures
\[
  \renewcommand{\arraystretch}{1.5}
  B=\begin{pmatrix}
B_{1}& & & \\
 &B_{2}& & \\
 & &\ddots& \\
 & & &B_{s}\\
\end{pmatrix}, \
\setlength{\arraycolsep}{-2pt}
E=\begin{pmatrix}
E_{1}& & & \\
&E_{2}& & \\
& &\vdots& \\
& &      & \ \ E_{s}\\
\end{pmatrix},\
\setlength{\arraycolsep}{3pt}
C=\begin{pmatrix}
C_{1}&C_{12}&\cdots&C_{1s}\\
C_{21}&C_{2}&\cdots&C_{2s}\\
\vdots&\vdots&\ddots&\vdots\\
C_{s1}&C_{s2}&\cdots&C_{s}\\
\end{pmatrix} ,
\]
while $F$ has the same block structure as that of  $E^T$.

For each subdomain $i$, $B_{i}$ denotes the matrix corresponding to
the interior variables and $C_{i}$ represents the matrix associated
with local interface variables, the matrices $E_i$ and $F_i$ denote
the couplings to local interface variables and the couplings from
local interface variables, respectively. A matrix $C_{ij}$ is a
nonzero matrix if and only if some interface variables of subdomain
$i$ are coupled with some interface variables of subdomain $j$.

After it is  reordered, the solution to
Equation (\ref{equ:resystem}) can  be found by solving
two intermediate problems
\begin{equation}\label{equ:solve}
  \left\{\begin{array}{lll}
  Sy=g-FB^{-1}f,  \\[2mm]
  Bx=f-Ey,
\end{array}\right.
\end{equation}
where $S=C-FB^{-1}E$ is the
\emph{Schur complement} of the coefficient matrix in (\ref{equ:resystem}).

Since $B$ and $E$ are block diagonal, the second equation in
(\ref{equ:solve}) can be solved efficiently once the vector $y$
becomes available. Many efforts have been devoted to develop
preconditioners for solving linear systems associated with $S$ in the
first equation.  Algebraic Recursive Multilevel Solvers (ARMS) is a
class Multilelvel ILU-type preconditioners
\cite{Saad-book2,Saad-Suchomel-ARMS} that consist of dropping small
entries of $S$ before applying an ILU factorization to it.  The SLR
preconditioner \cite{m7} developed more recently approximates $S^{-1}$
by the sum of $C^{-1}$ and a low-rank correction term. Here the
low-rank correction term is computed by exploiting the eigenvalue
decay property of $S^{-1}-C^{-1}$.  A relative to SLR is the
Multilevel Schur Low-Rank (MSLR) preconditioner \cite{m8} which
approximates $S^{-1}$ by applying the same idea as in SLR recursively
in order to address the scalability issue. Finally GMSLR \cite{m9} was
developed as a generalization of MSLR to nonsymmetric systems.

\section{The PSLR preconditioner}\label{sec:pre}
In this section, we first derive a power series expansion of $S^{-1}$,
and then discuss  low-rank correction techniques whose goal is to
improve its approximation accuracy.

\subsection{Power series expansion of the inverse of the
  Schur complement}\label{sec:expansion}
The proposed power series expansion is applied to a splitting form of
$S$ rather than $S$ itself.
Specifically, we first write the Schur complement $S$
as the difference of two matrices:
\begin{eqnarray}\label{matrixS}
 S = C_{0}-E_{s},
\end{eqnarray}
where
$$C_{0}=\text{diag}\begin{pmatrix}
                          C_{1}, & C_{2}, & \ldots, & C_{s} \\
                        \end{pmatrix}
$$
is the block diagonal part of $C$ and $E_s=C_0-S$.
Note that $E_s = (C_0-C) + F B\inv E $.
Then we have
\begin{equation}\label{invS0}
S^{-1}= (I-C_{0}^{-1}E_{s})^{-1}C_{0}^{-1}.
\end{equation}
Next, we simply apply a $(m+1)$-term power series expansion of
$(I-C^{-1}_{0}E_s)^{-1}$ to obtain  the following approximation to
$S^{-1}$
\begin{equation}
S^{-1} \approx \sum_{i=0}^{m}(C_{0}^{-1}E_{s})^{i}C_{0}^{-1}.
\label{eq:approx1}
\end{equation}
One immediate advantage of using \eqref{eq:approx1} is that the
application of $\sum_{i=0}^{m}(C_{0}^{-1}E_{s})^{i}C_{0}^{-1}$ on a
vector only involves linear system solutions associated with $C_0$
and $B$, as well as
matrix vector multiplications associated with $E$ and $F$.
The block diagonal structures in these three matrices make these operations
extremely efficient.

Using results with standard norms it is straightforward to
prove the following  proposition which  analyzes
the approximation accuracy of \eqref{eq:approx1}.
\begin{proposition}\label{t1}
  If the spectral radius of $C_{0}^{-1}E_{s}$ satisfies
  $\rho(C_{0}^{-1}E_{s})<1$,
then
\begin{equation}\label{Sexp}
S^{-1} =\sum_{i=0}^{m}(C_{0}^{-1}E_{s})^{i}C_{0}^{-1}+ R,
\end{equation}
where the error matrix
\begin{equation}\label{equR}
R=\sum_{i=m+1}^{\infty}(C_{0}^{-1}E_{s})^{i}C_{0}^{-1}
\end{equation}
satisfies
\begin{equation}\label{normR}
\Vert R \Vert\leq\frac{\Vert C_{0}^{-1}E_{s} \Vert^{m+1}\Vert C_{0}^{-1}\Vert}{1-
\Vert C_{0}^{-1}E_{s} \Vert}.
\end{equation}
\end{proposition}

A large class of matrices satisfy the condition
$\rho(C_{0}^{-1}E_{s})<1$ as required in Proposition \ref{t1}. For
example, we can show that $\rho(C_{0}^{-1}E_{s})<1$ holds whenever $A$
is symmetric positive definite (SPD) and its (2,2)-block $C$ is
diagonally dominant in the next lemma.
\begin{lemma}\label{l2}
If $A$ in (\ref{equ:resystem}) is SPD, then
\begin{equation}\label{rho}
\lambda(C_{0}^{-1}E_{s})<1.
\end{equation}
Moreover, if the (2,2)-block $C$ of $A$ is diagonally dominant, then
\begin{equation}\label{rho0}
\lambda(C_{0}^{-1}E_{s})>-1.
\end{equation}
Here $\lambda(\cdot)$ denotes any eigenvalue of a matrix.
\end{lemma}

\begin{proof}
  Since $A$ is SPD, $S$ and $C_0$ are also SPD and
  $C_{0}^{-\frac{1}{2}}SC_{0}^{-\frac{1}{2}}$ is
  SPD. Thus the eigenvalues of $C_0\inv S$, which is similar to
  $C_{0}^{-\frac{1}{2}}SC_{0}^{-\frac{1}{2}}$, are all real and
  positive. Moreover,
\begin{equation}\label{CE}
C_{0}^{-1}E_{s}=C_{0}^{-1}(C_{0}-S)=I-C_{0}^{-1}S,
\end{equation}
and this shows that
$
\lambda(C_{0}^{-1}E_{s}) <1 $.

Now we prove the second part of this lemma. Let
\[ C_g = C - C_0 ,
  \]
  which is the matrix $C$ stripped off its diagonal blocks, and note that
  $C_0 - C_g = 2 C_0 - C$.
  Then we have
\begin{eqnarray*}
  2I-C_{0}^{-1}S &=&C_{0}^{-1}(2C_{0}-S) \\
&=&C_{0}^{-1}(C_{0}-C_{g}+E^{T}B^{-1}E),
\end{eqnarray*}
which is similar to
$$\Phi=C_{0}^{-\frac{1}{2}}(C_{0}-C_{g}+E^{T}B^{-1}E)C_{0}^{-\frac{1}{2}}.$$
Since $C$ is a diagonally dominant matrix, the matrix $C_{0}-C_{g}$
is also diagonally dominant. This results in the symmetric positive
definiteness of $\Phi$. Hence, the eigenvalues of $2I-C_{0}^{-1}S$ are
all positive, leading to
\begin{eqnarray}\label{CS}
\lambda(C_{0}^{-1}S)<2.
\end{eqnarray}
This along with (\ref{CE}) yields the desired result:
$\lambda(C_{0}^{-1}E_{s})=1-\lambda(C_{0}^{-1}S)>-1$.
\end{proof}

\medskip Lemma \ref{l2} shows that $\rho(C_{0}^{-1}E_{s})<1$, when $A$
is SPD and $C$ is diagonally dominant.  As an example, we depict the
eigenvalues of $C_{0}^{-1}E_{s}$ in Figure \ref{fig:cos1} for a $3D$
discretized Laplacian matrix $A$ on a $20^3$ grid and the number $s$
of subdomains is set to $s=5$. It is easy to see that the absolute
values of all the eigenvalues of $C_{0}^{-1}E_{s}$ are smaller than
$1$. 


\begin{figure}[ptbh]
	\begin{tabular}
		[r]{c}%
		\includegraphics[height=1.250in]{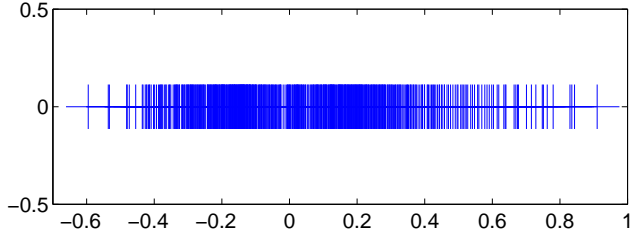}
	\end{tabular}
	\caption{Eigenvalues (\textcolor{blue}{'+'}) of $C_{0}^{-1}E_{s}$ for a  3D Laplacian matrix discretized on a $20^3$ grid with the zero Dirichlet boundary condition where the number of subdomains $s =5$.}%
	\label{fig:cos1}%
\end{figure}
\subsection{Low-rank approximations of $S^{-1}$}\label{low-rank}

The power series expansion of $S^{-1}$ in Section \ref{sec:expansion}  only provides a rough approximation to $S^{-1}$, especially when $m$ is small or/and $\rho(C_{0}^{-1}E_{s})$ is slightly smaller than $1$. In this section, we will consider some low-rank
correction techniques to improve the accuracy of this approximation. In addition, we will also consider the case when $\rho(C_{0}^{-1}E_{s})>1$.

{First define
\begin{equation}\label{equR1}
\widehat{R}=S^{-1}-\sum_{i=0}^{m}(C_{0}^{-1}E_{s})^{i}C_{0}^{-1}.
\end{equation}
Notice that when $\rho(C_{0}^{-1}E_{s})<1$, $\widehat{R}$ is equal to
the matrix $R$ defined in (\ref{equR}).

Then we have
\begin{equation}\label{Sexp1}
S^{-1} =\sum_{i=0}^{m}(C_{0}^{-1}E_{s})^{i}C_{0}^{-1}+ \widehat{R}.
\end{equation}}

Let
\begin{equation}\label{err1}
E_{rr}(m):=S\widehat{R} \in \mathbb{R}^{q\times q}.
\end{equation}
Based on (\ref{Sexp1}) we get
\begin{eqnarray}\label{err0}
I&= & S\sum_{i=0}^{m}(C_{0}^{-1}E_{s})^{i}C_{0}^{-1} +S\widehat{R}\\ \nonumber
&=& S\sum_{i=0}^{m}(C_{0}^{-1}E_{s})^{i}C_{0}^{-1} +E_{rr}(m),
\end{eqnarray}
which leads to
\begin{equation}\label{s1}
 S^{-1}= \bigg[\sum_{i=0}^{m}(C_{0}^{-1}E_{s})^{i}C_{0}^{-1}\bigg](I-E_{rr}(m))^{-1}.
\end{equation}
Here, we assume $I-E_{rr}(m)$ is nonsingular.

Equation \eqref{s1} provides another way to approximate $S^{-1}$. If a $r_k$-step Arnoldi procedure is performed on $E_{rr}(m)$, $E_{rr}(m)$ can be approximated by
\begin{equation}\label{schur}
E_{rr}(m)\approx V_{r_k}H_{r_k}V_{r_k}^{T},
\end{equation}
where $V_{r_k}\in \mathbb{R}^{q\times r_k}$ has orthonormal columns and $H_{r_k}
=V_{r_k}^TE_{rr}(m)V_{r_k}\in \mathbb{R}^{r_k\times r_k}$
is an upper Hessenberg matrix whose eigenvalues can be used to
approximate the largest eigenvalues of $E_{rr}(m)$.
For a given $m$, it can be justified that the Frobenius norm
$\|E_{rr}(m)-V_{r_k}H_{r_k}V_{r_k}^{T}\|_F$ decreases monotonically
as $r_k$ increases \cite{Saad-book2}. As a result, $V_{r_k}H_{r_k}V_{r_k}^{T}$
approximates $E_{rr}(m)$ more accurately as $r_k$ increases.

Combining (\ref{s1}) with (\ref{schur}) gives rise to
\begin{eqnarray}\label{Sinv0}
S^{-1}&\approx & \bigg[\sum_{i=0}^{m}(C_{0}^{-1}E_{s})^{i}C_{0}^{-1}\bigg]
(I-V_{r_k}H_{r_k}V_{r_k}^{T})^{-1}\\  \nonumber
&=& \bigg[\sum_{i=0}^{m}(C_{0}^{-1}E_{s})^{i}C_{0}^{-1}\bigg](I + V_{r_k}[(I-H_{r_k})^{-1}-I]V_{r_k}^{T}) \\ \nonumber
&=& \bigg[\sum_{i=0}^{m}(C_{0}^{-1}E_{s})^{i}C_{0}^{-1}\bigg](I + V_{r_k}G_{r_k}V_{r_k}^{T}),
\end{eqnarray}
where
$G_{r_k}=(I-H_{r_k})^{-1}-I\in \mathbb{R}^{r_k\times r_k}$. In the above process, we utilize the Sherman-Morrison-Woodbury
formula to derive the expression of $(I-V_{r_k}H_{r_k}V_{r_k}^{T})^{-1}$.

Thus, the final approximation to $S\inv$ takes the form:
\begin{equation}\label{Sinv}
  S_{\text{app}}^{-1}=\bigg[\sum_{i=0}^{m}(C_{0}^{-1}E_{s})^{i}C_{0}^{-1}\bigg](I + V_{r_k}G_{r_k}V_{r_k}^{T}) .
\end{equation}

\subsubsection{Approximation accuracy analysis}\label{approximation}
In this section, we quantify the approximation accuracy of $S_{\text{app}}^{-1}$ in terms of $m$ and $r_k$.
The next theorem first shows the relation between the eigenvalue decay rate of $E_{rr}(m)$ and the number of the power series expansion $m+1$.
\begin{theorem}\label{t2}
For any matrix $A$, the matrix $E_{rr}(m)$ in (\ref{err1}) can be rewritten as
\begin{equation}\label{err}
E_{rr}(m)=(E_{s}C_{0}^{{-1}})^{m+1}.
\end{equation}
\end{theorem}

\begin{proof}
Combining (\ref{matrixS}) with (\ref{err0}), we have
\begin{eqnarray*}
E_{rr}(m)&=&I-S\bigg[\sum_{i=0}^{m}(C_{0}^{-1}E_{s})^{i}\bigg]C_{0}^{-1}\\
&=& I-(C_{0}-E_{s})\bigg[\sum_{i=0}^{m}(C_{0}^{-1}E_{s})^{i}\bigg]C_{0}^{-1}\\
&=& I-C_{0}\bigg[\sum_{i=0}^{m}(C_{0}^{-1}E_{s})^{i}\bigg]C_{0}^{-1}+E_{s}
\bigg[\sum_{i=0}^{m}(C_{0}^{-1}E_{s})^{i}\bigg]C_{0}^{-1}\\
&=&I-\sum_{i=0}^{m}(E_{s}C_{0}^{-1})^{i}+\sum_{i=1}^{m+1}(E_{s}C_{0}^{-1})^{i}\\
&=&(E_{s}C_{0}^{-1})^{m+1} .
\end{eqnarray*}
This completes the proof.
\end{proof}

Theorem \ref{t2} shows that the eigenvalues
of $E_{rr}(m)$ decay faster as $m$ increases. In fact, the eigenvalue decay rate of  $E_{rr}(m)$ is $m+1$ times faster than that of  $E_{rr}(0)$. Figure \ref{fig:decay0} and Figure \ref{fig:pde900} further justify Theorem \ref{t2} numerically on one symmetric 3D discretized Laplacian matrix (Figure \ref{fig:decay0}) and one non-symmetric \texttt{pde900} matrix from the SuiteSparse collection~\cite{Davis} (Figure \ref{fig:pde900}). {The spectral radius of $E_{s}C_{0}^{-1}$ are equal to $0.9537$ and $0.8117$, respectively, in these two examples}. As can be seen from Figure \ref{fig:decay0} and Figure \ref{fig:pde900}, the eigenvalues of $E_{rr}(m)$ get more clustered around the origin when a larger $m$ is used. Here, small values of $m$,
i.e., $m=0,1,2,3$, are tested.

\begin{figure}[ptbh]
\centering\tabcolsep=0mm
	\begin{tabular}
		[r]{cc}%
		\includegraphics[height=0.75in]{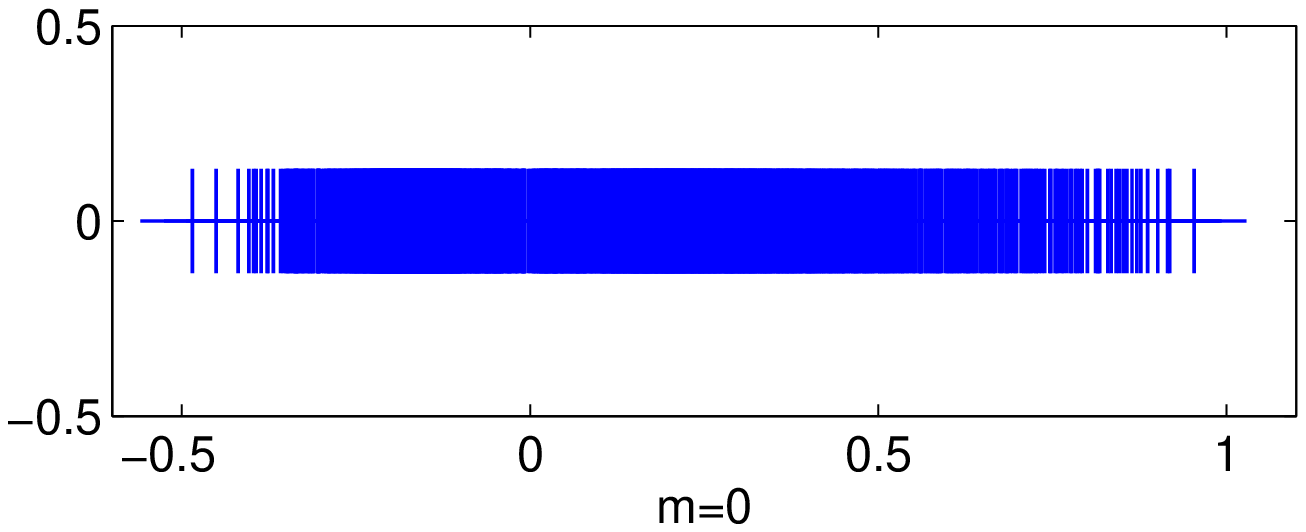}&
		\includegraphics[height=0.75in]{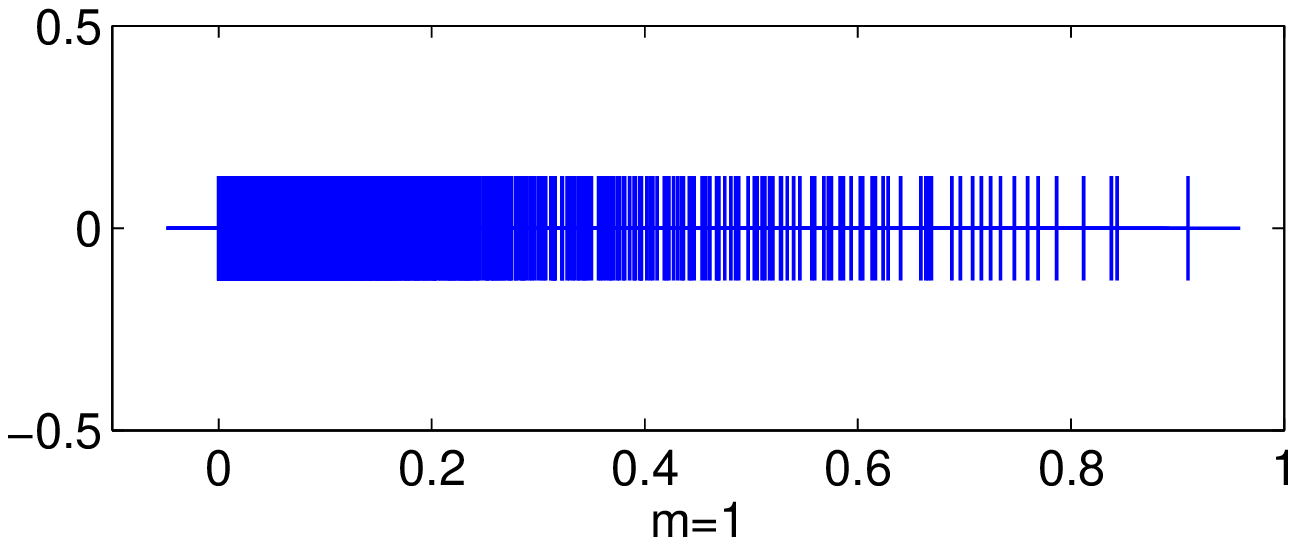}\\
		\includegraphics[height=0.75in]{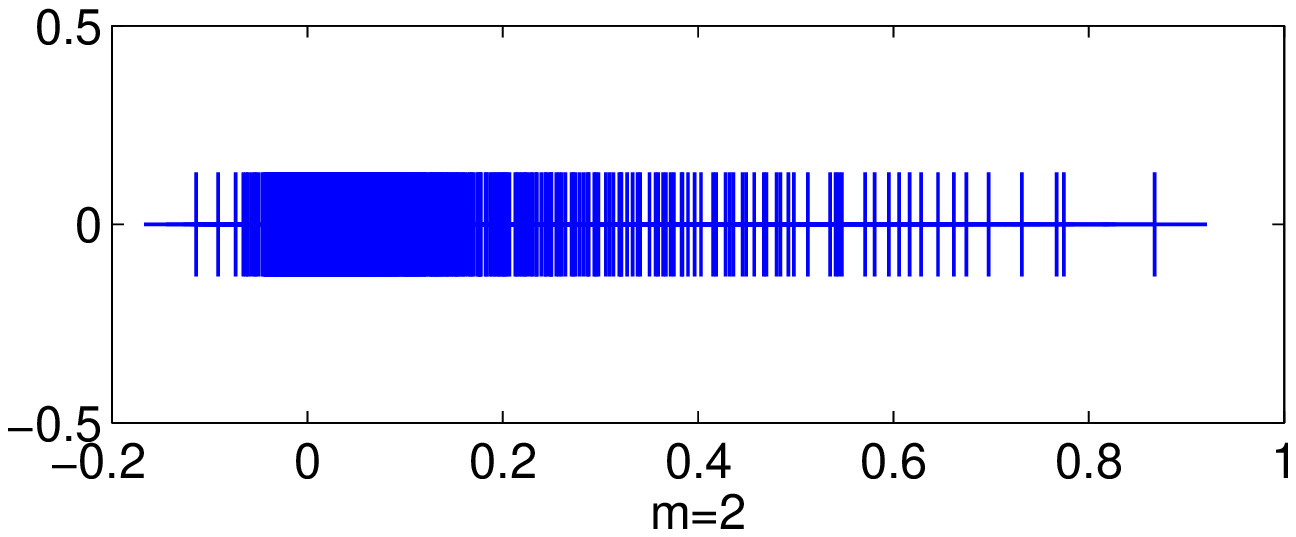}&
		\includegraphics[height=0.75in]{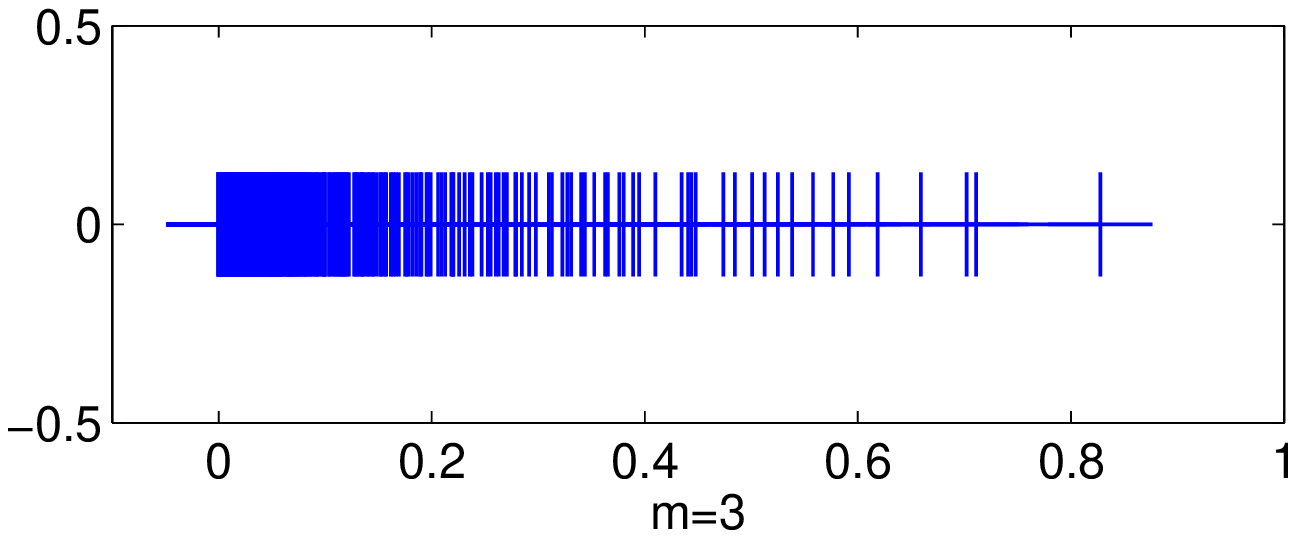}
	\end{tabular}
	\caption{Spectrum of $E_{rr}(m)$ with different values of $m$,
          where the matrix $A$ is taken as a $3D$ Laplacian matrix
          discretized on a $20^3$ grid with the zero Dirichlet
          boundary condition. In this test, the number of subdomains
          is chosen as $s=5$ and
          {$\rho(E_{s}C_{0}^{-1})=0.9537$. Here, a blue {`$+$'} denotes
            an eigenvalue of $E_{s}C_{0}^{-1}$. }}%
	\label{fig:decay0}%
\end{figure}
\begin{figure}[ptbh]
\centering\tabcolsep=0mm
	\begin{tabular}
		[r]{cc}%
		\includegraphics[height=1.5in]{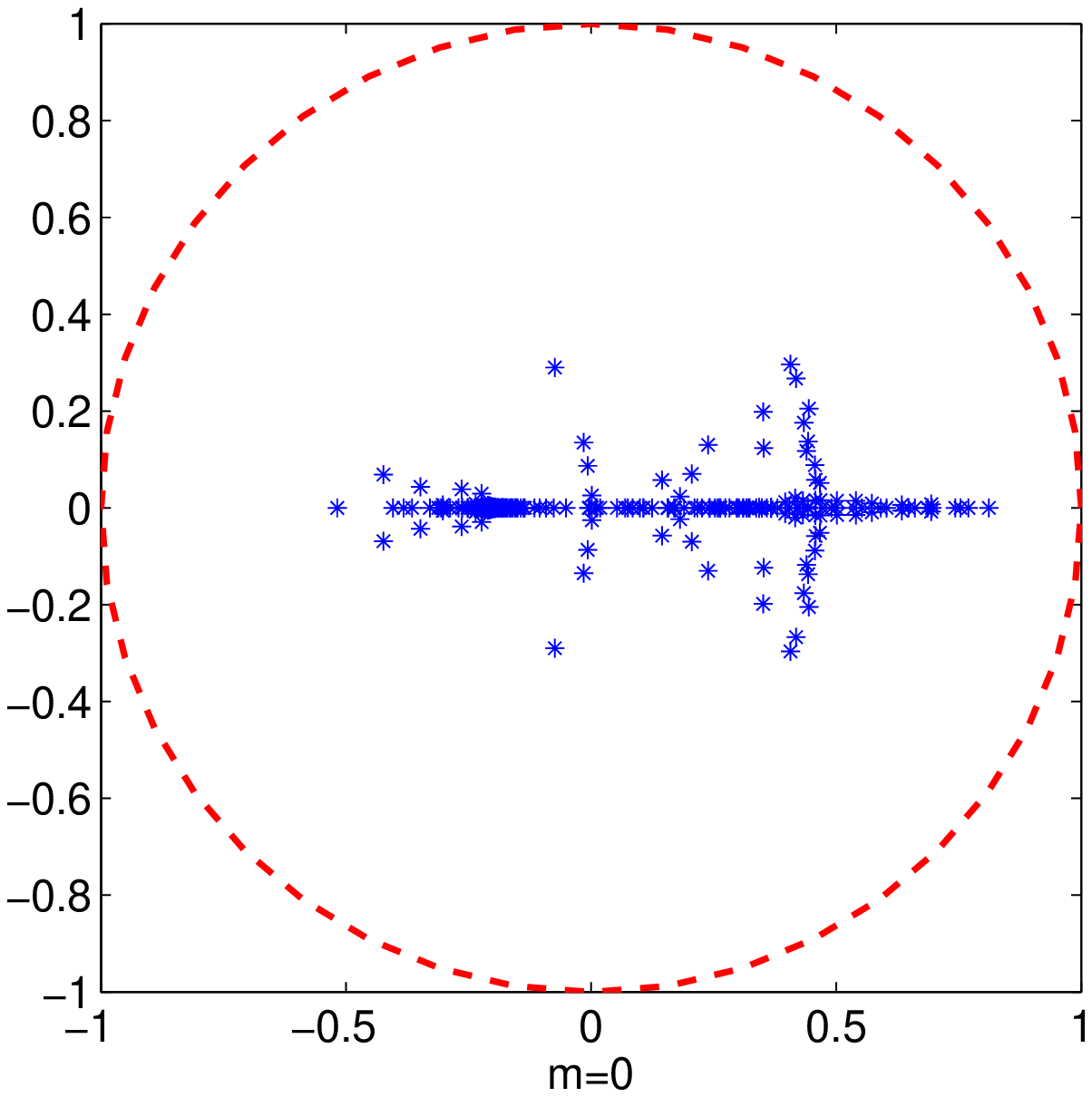}&
		\includegraphics[height=1.5in]{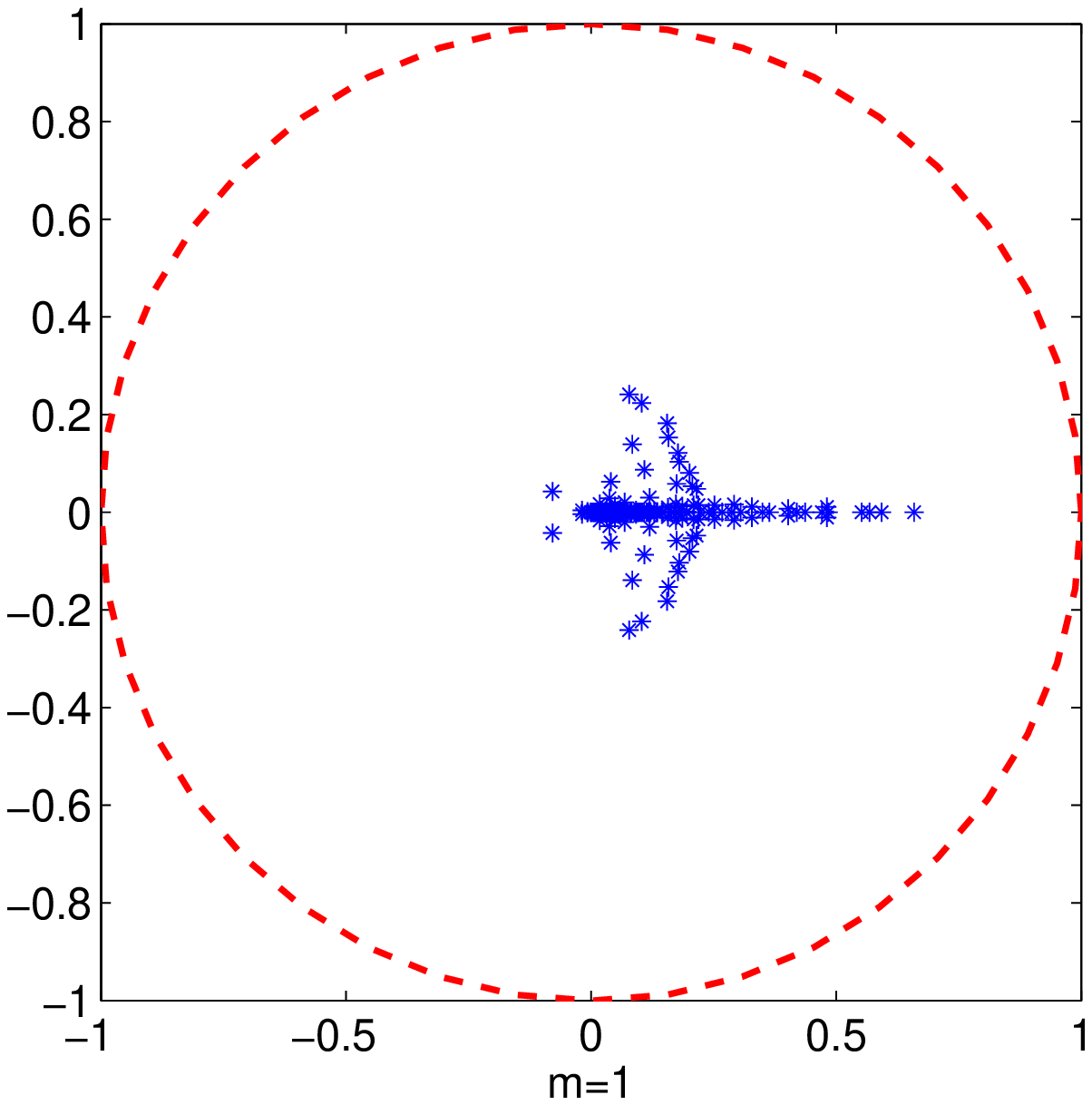}\\
		\includegraphics[height=1.5in]{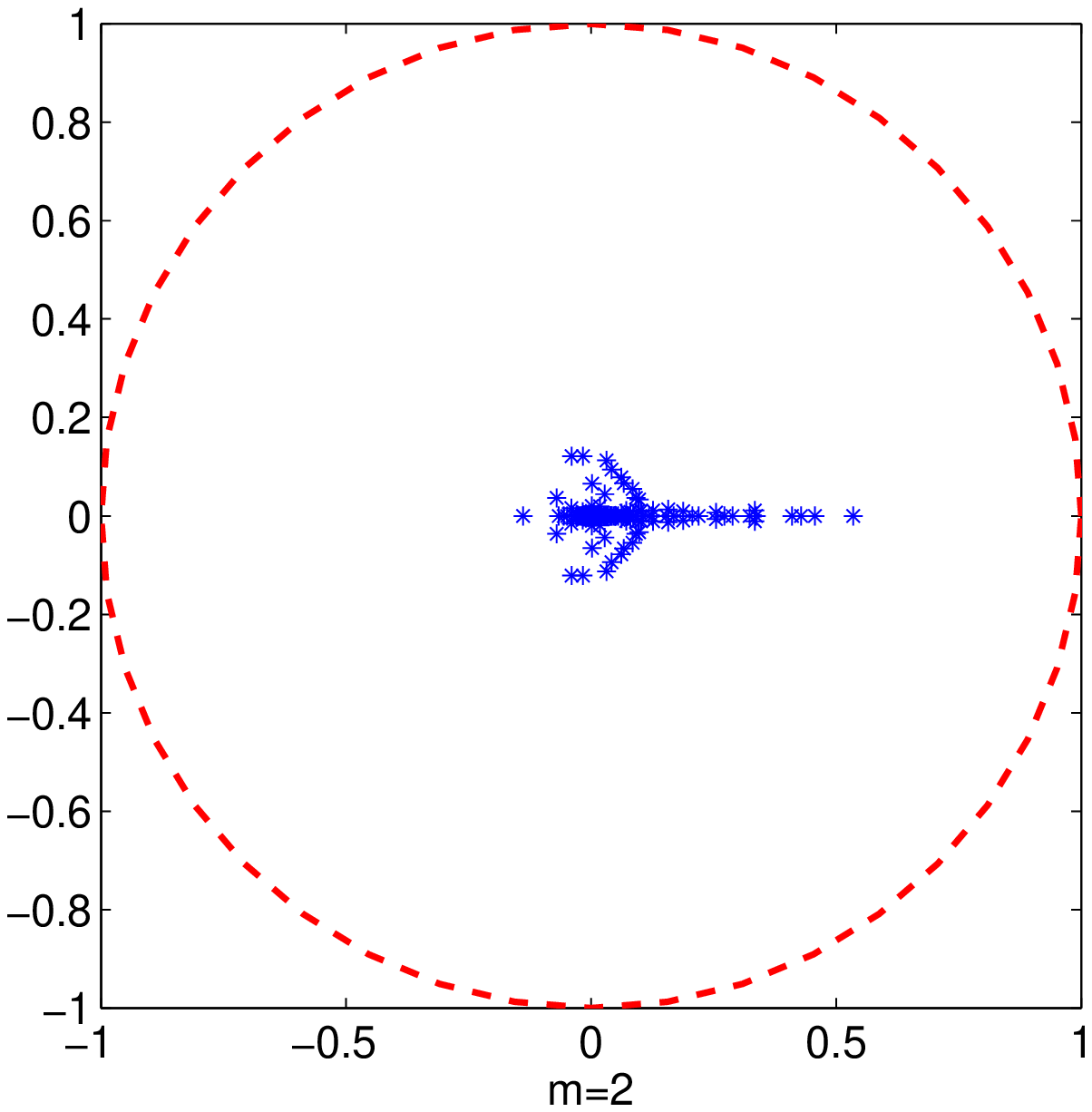}&
		\includegraphics[height=1.5in]{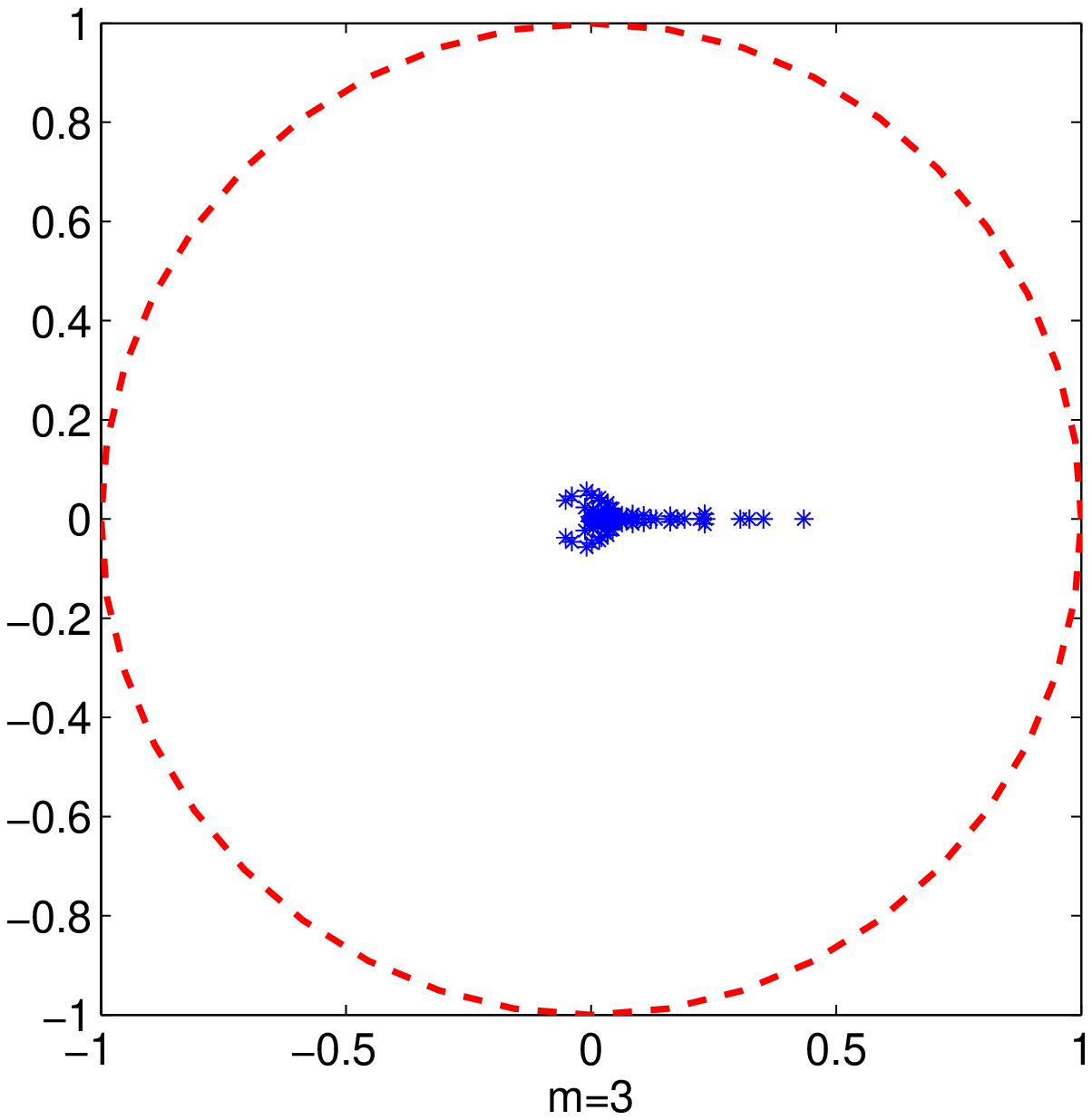}
	\end{tabular}
	\caption{{Spectrum of $E_{rr}(m)$ with different $m$, where
            the matrix $A$ is the non-symmetric
            \texttt{pde900} taken from the SuiteSparse
            collection~\cite{Davis}.  The test matrix has the
            dimension of $900\times 900$ and is indefinite. The number
            of subdomains used in the partition is $s=5$ and
            $\rho(E_{s}C_{0}^{-1})=0.8117$. Here, a blue {`$\ast$'}
            denotes an eigenvalue of $E_{s}C_{0}^{-1}$ and the red
            dashed circle has radius $1$.}}%
	\label{fig:pde900}%
\end{figure}


We then consider two indefinite matrices. The first one is the $3D$
shifted discretized Laplacian matrix (Figure \ref{fig:decay1}) and the
second one is the non-symmetric \texttt{young1c} matrix from the
SuiteSparse collection \cite{Davis} (Figure \ref{fig:pde2961}). The
indefiniteness causes the spectral radius of $E_{s}C_{0}^{-1}$ greater
than $1$ in both tests.  But as can be seen from Figures
\ref{fig:decay1}-\ref{fig:pde2961}, only a few eigenvalues have
modulus greater than $1$. As a result, the majority of the eigenvalues
still get clustered around the origin as $m$ increases. Based on this
property, we can show that the approximation accuracy of \eqref{Sinv}
can be improved as $m$ increases under mild conditions. In contrast,
the classical Neumann series expansion \eqref{eq:approx1} will diverge
in this case.

\begin{figure}[ptbh]
\centering\tabcolsep=0mm
	\begin{tabular}
		[r]{cc}%
		\includegraphics[height=0.75in]{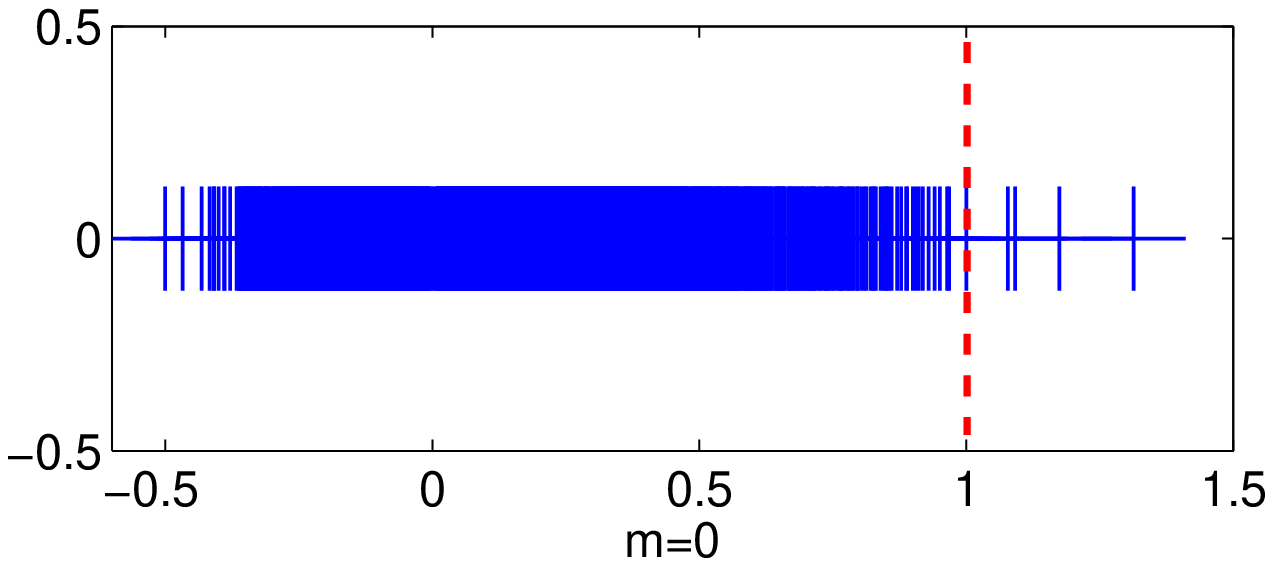}&
		\includegraphics[height=0.75in]{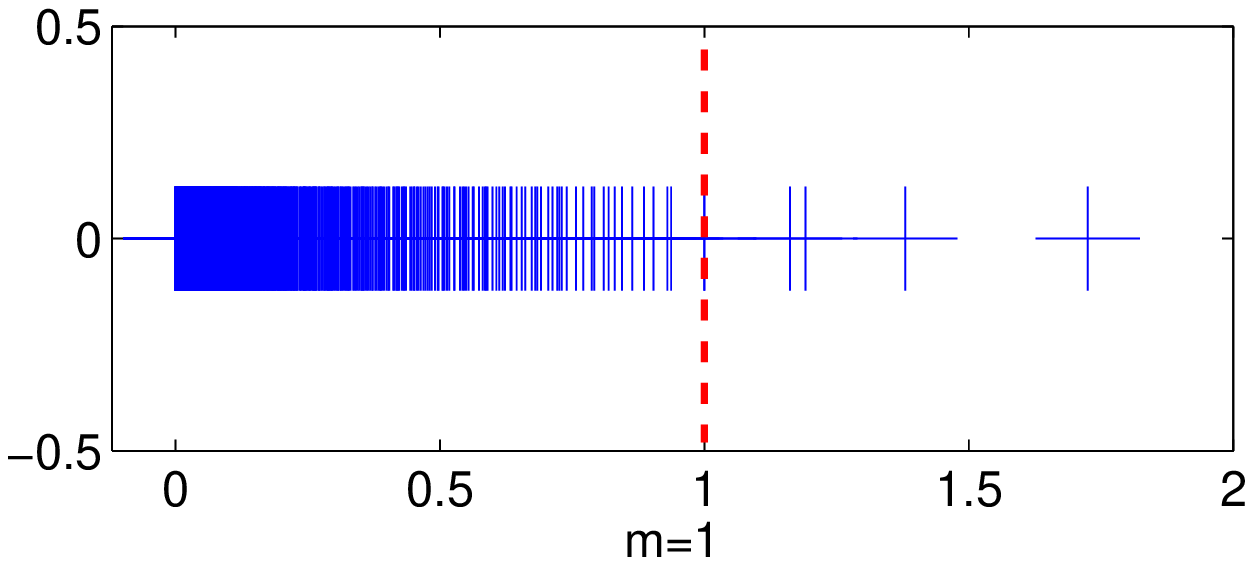}\\
		\includegraphics[height=0.75in]{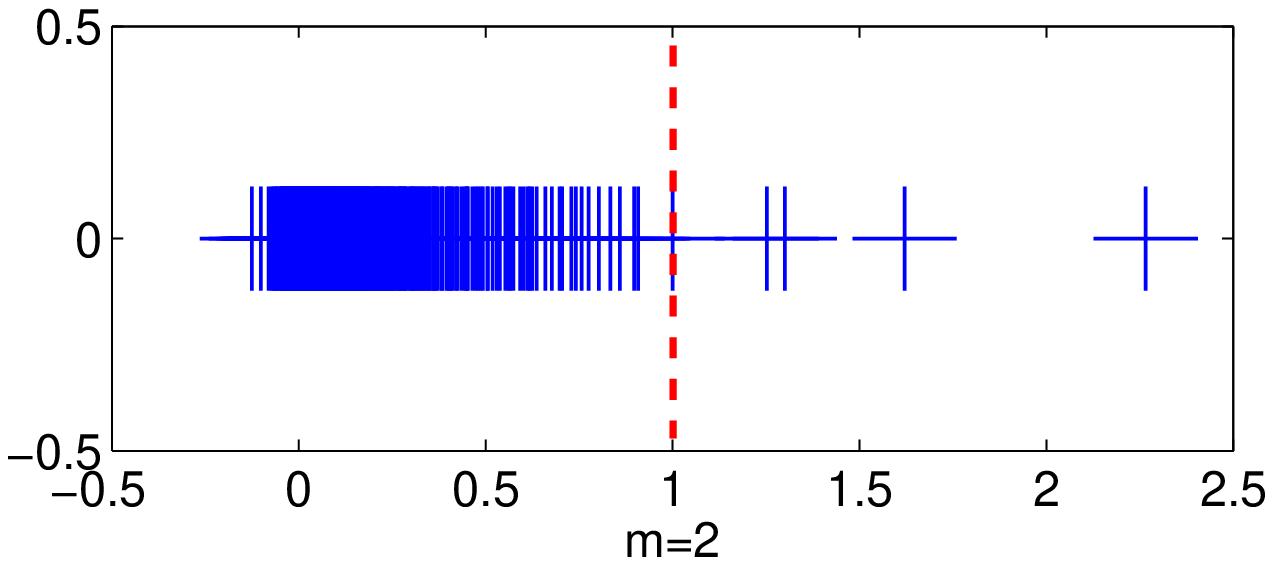}&
		\includegraphics[height=0.75in]{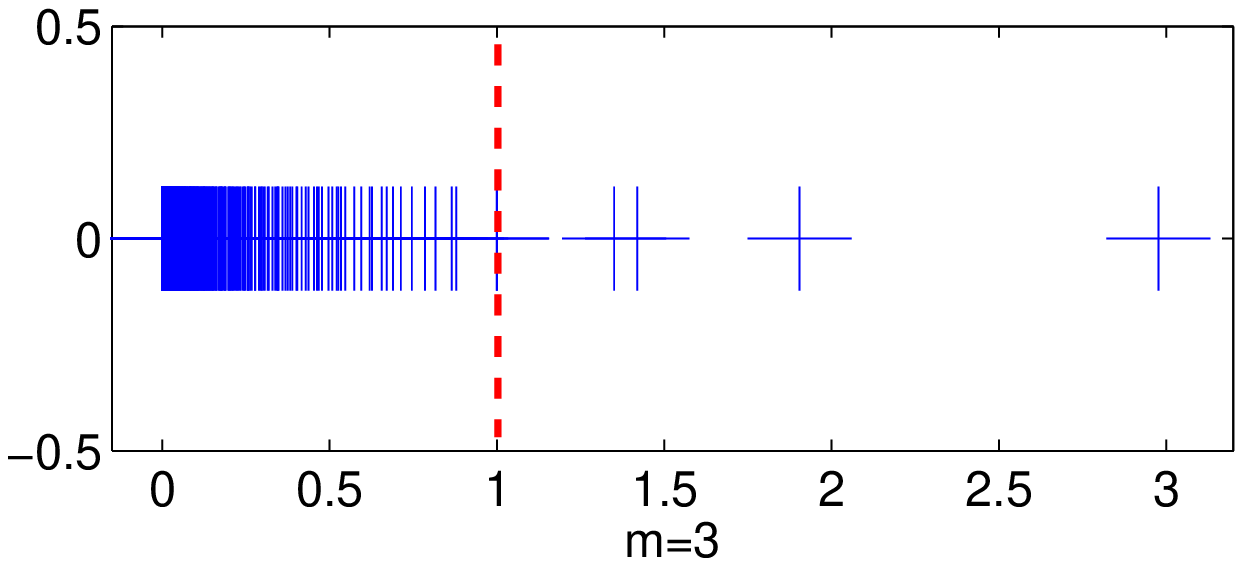}
	\end{tabular}
	\caption{Spectrum of $E_{rr}(m)$ with different $m$, where the
          matrix $A$ is taken as a shifted $3D$ Laplacian matrix
          discretized on a $20^3$ grid with the zero Dirichlet
          boundary condition and the number of subdomains is chosen as
          $s=5$. In this test, $A$ is indefinite and has $7$ negative
          eigenvalues and $E_{s}C_{0}^{-1}$ has $4$ eigenvalues larger
          than $1$. Here, a blue {`+'} denotes an eigenvalue of
          $E_{s}C_{0}^{-1}$. }%
	\label{fig:decay1}%
\end{figure}
\begin{figure}[ptbh]
\centering\tabcolsep=0mm
	\begin{tabular}
		[r]{cc}%
		\includegraphics[height=1.5in]{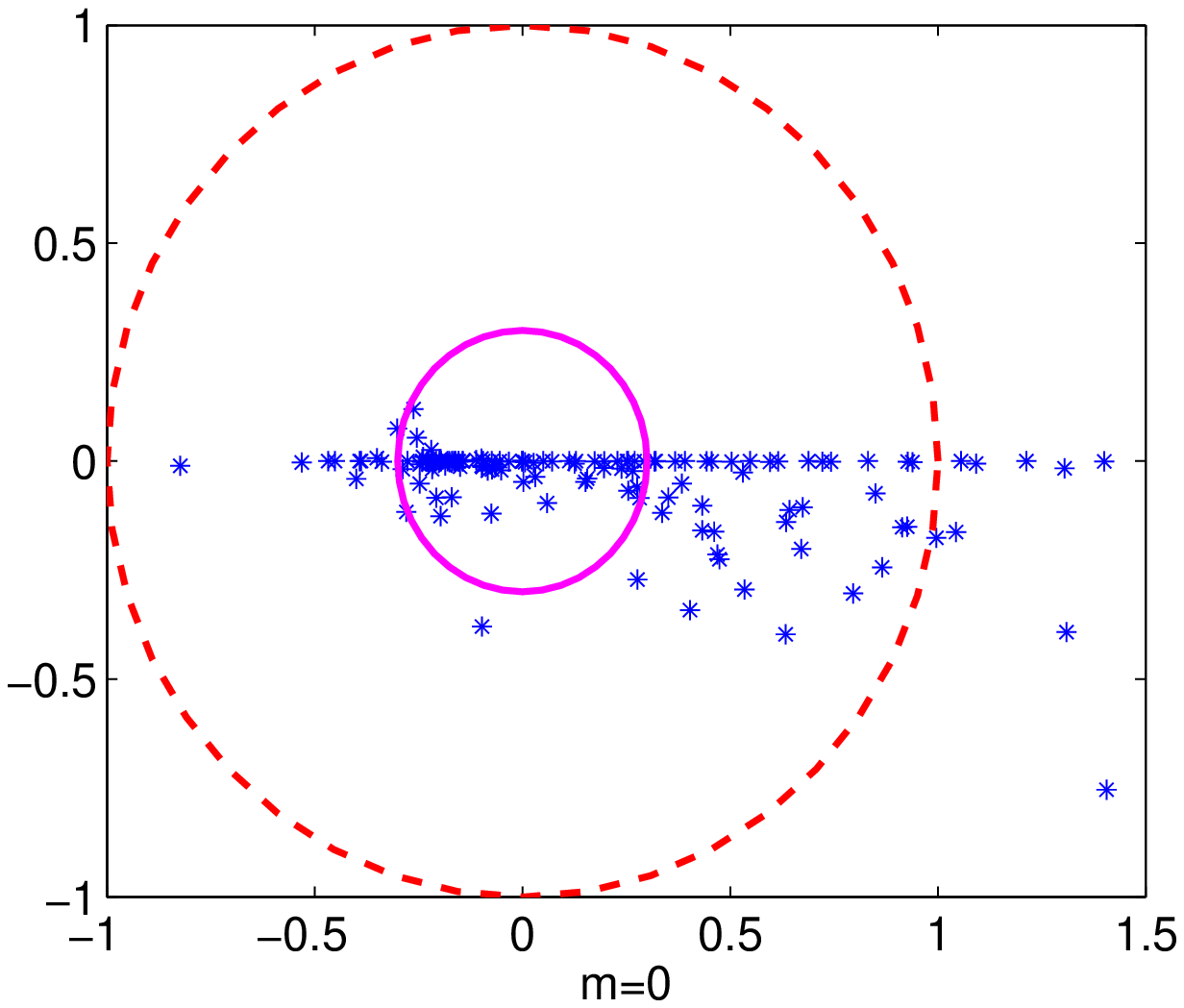}&
		\includegraphics[height=1.5in]{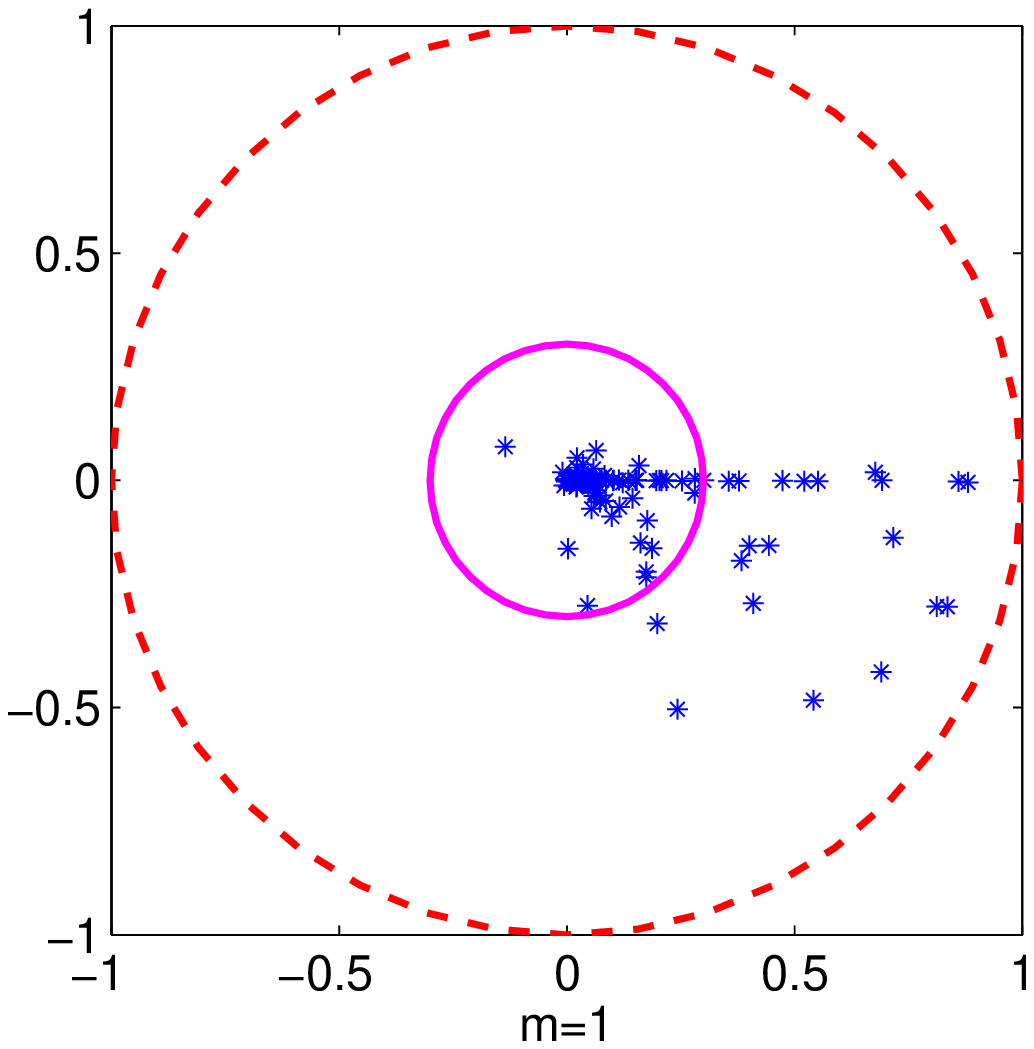}\\
		\includegraphics[height=1.5in]{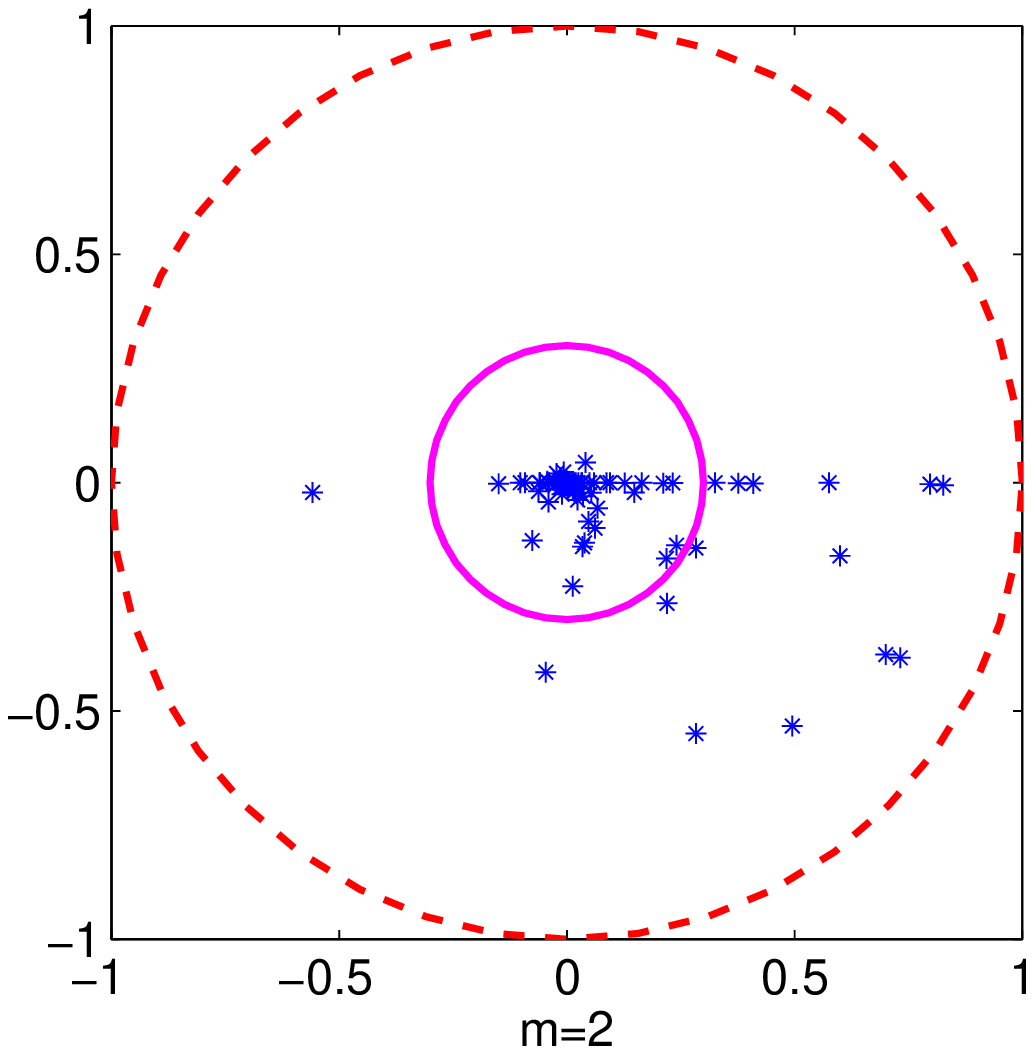}&
		\includegraphics[height=1.5in]{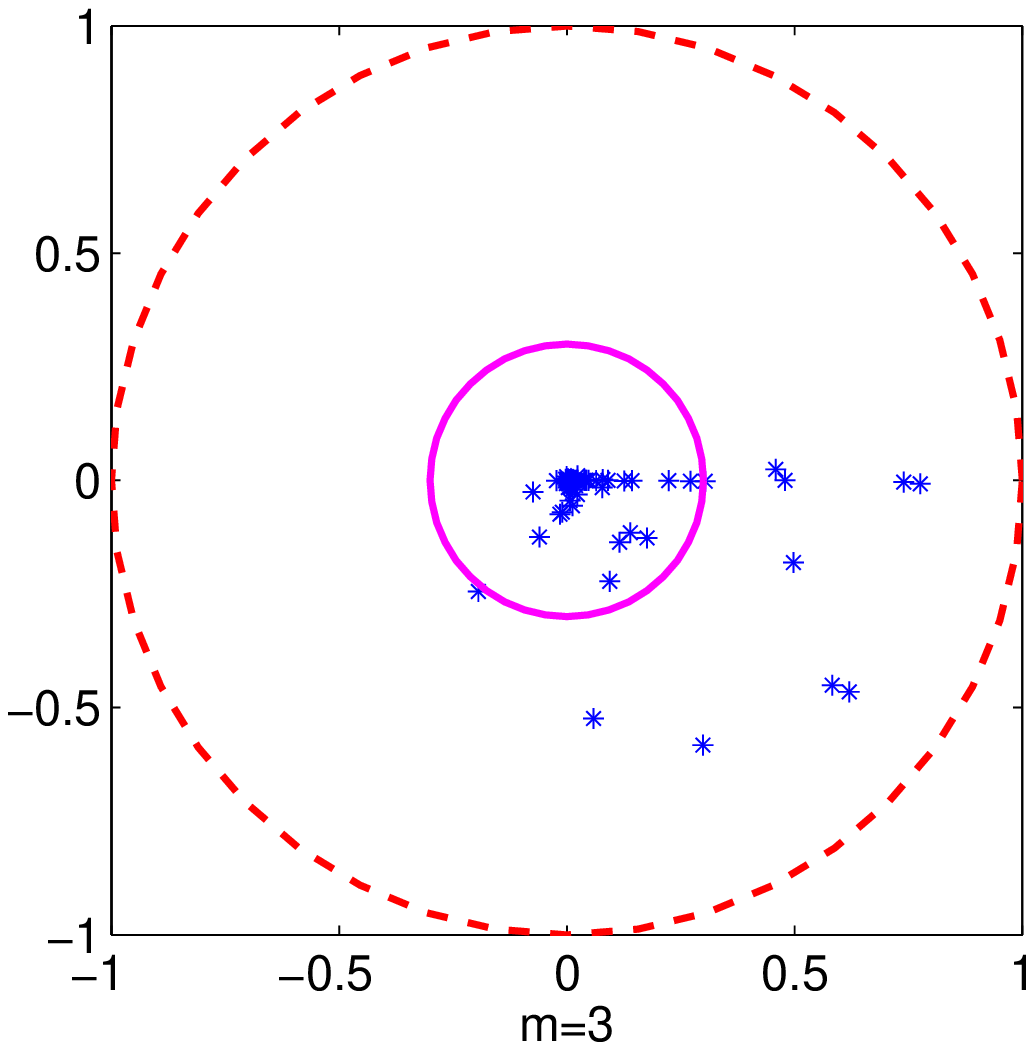}
	\end{tabular}
	\caption{{Spectrum of $E_{rr}(m)$ with different $m$, where
            the matrix $A$ is the non-symmetric \texttt{young1c}
            matrix from the SuiteSparse collection~\cite{Davis}. This
            matrix has the dimension of $841\times 841$ and the number
            of subdomains used in the partition is
            $s=5$. $E_{s}C_{0}^{-1}$ has $9$ eigenvalues with modulus
            larger than $1$ which are shown outside the red dashed
            circle in the top-left subfigure. Only the eigenvalues of
            $E_{rr}(m)$ within the unit red dashed circle are shown
            for the cases $m=1,2,3$. Here, a blue {` $\ast$'} denotes an
            eigenvalue of $E_{s}C_{0}^{-1}$, the red dashed circle has
            radius $1$ while the pink solid circle has radius $0.3$.}}%
	\label{fig:pde2961}%
\end{figure}

We first prove {an upper bound of the relative approximation accuracy
  of $S_{\text{app}}^{-1}$ to $S^{-1}$ in the next proposition.}
\begin{proposition}\label{l3}
For any matrix norm $\parallel\cdot\parallel$, the approximation accuracy of $S_{\text{app}}^{-1}$ to $S^{-1}$ satisfies the following inequality
\begin{equation}\label{Serr333}
  \frac{\parallel S^{-1}-S_{\text{app}}^{-1}\parallel}{\parallel S^{-1}\parallel}
  \leq \parallel X(m,r_k)\parallel\parallel Z(r_k)^{-1} \parallel,
\end{equation}
where
\begin{equation}\label{denote}
  X(m,r_k)=E_{rr}(m)-V_{r_k}H_{r_k}V_{r_k}^{T},~Z(r_k)=I-V_{r_k}H_{r_k}V_{r_k}^{T}.
\end{equation}
\end{proposition}

\begin{proof}
From (\ref{s1}) and (\ref{Sinv0}), we have
\begin{eqnarray}\label{diff}
  S^{-1}-S_{\text{app}}^{-1}
  &=&\bigg[\sum_{i=0}^{m}(C_{0}^{-1}E_{s})^{i}C_{0}^{-1}
  \bigg]\big[(I-E_{rr}(m))^{-1}-(I-V_{r_k}H_{r_k}V_{r_k}^{T})^{-1}\big] \\ \nonumber
  &=&\bigg[\sum_{i=0}^{m}(C_{0}^{-1}E_{s})^{i}C_{0}^{-1}\bigg]
  \big[\big(I-V_{r_k}H_{r_k}V_{r_k}^{T}-X(m,r_k)\big)^{-1}-(I-V_{r_k}H_{r_k}V_{r_k}^{T})^{-1}\big]  \\ \nonumber
  &=&\bigg[\sum_{i=0}^{m}(C_{0}^{-1}E_{s})^{i}C_{0}^{-1}\bigg]\bigg[\big(Z(r_k)-X(m,r_k)\big)^{-1}-Z(r_k)^{-1}\bigg] \\ \nonumber
  &=& \bigg[\sum_{i=0}^{m}(C_{0}^{-1}E_{s})^{i}C_{0}^{-1}\bigg]\big(Z(r_k)-X(m,r_k)\big)^{-1}X(m,r_k)Z(r_k)^{-1} \\ \nonumber
  &=& \bigg[\sum_{i=0}^{m}(C_{0}^{-1}E_{s})^{i}C_{0}^{-1}\bigg](I-E_{rr}(m))^{-1}X(m,r_k)Z(r_k)^{-1}  \\ \nonumber
   &=& S^{-1}X(m,r_k)Z(r_k)^{-1} .
\end{eqnarray}
Using a matrix norm, this yields
\[
  \parallel S^{-1}-S_{\text{app}}^{-1}\parallel\leq \parallel S^{-1}\parallel\parallel X(m,r_k)
\parallel\parallel Z(r_k)^{-1} \parallel, \]
from which \eqref{Serr333} follows.
\end{proof}

Next, we provide two numerical experiments to illustrate Proposition
\ref{l3}. The Frobenius norm is employed for both tests and we denote by $\Delta(m,r_k)$ the upper bound $\parallel X(m,r_k)\parallel\parallel Z(r_k)^{-1} \parallel$ in Proposition \ref{l3}.
The first test is a 3D Laplacian matrix and the second one
is a shifted 3D Laplacian matrix.
Both matrices have size of
$2,000\times 2,000$. In the tests, we fix $r_k=15$ and $s=5$ and change
numbers of terms used in the power series expansion from $m=3$ to
$m=5$. For the Laplacian matrix, we have $\Delta(3,15)=0.49$ when
$m=3$ and $\Delta(5,15)=0.15$ when $m=5$. For the shifted Laplacian
matrix, $E_{rr}(m)$ has $11$ eigenvalues with modulus
  larger than $1$. Since $r_k$ is larger than $11$, when $m=3$ and $m=5$, we have $\Delta(3,15)=0.72$
and $\Delta(5,15)=0.665$, respectively. These two tests verify that
the approximation is more accurate if $m$ increases as long as the rank $r_k$
is larger than the number of the eigenvalues of $E_{rr}$ with modulus greater than $1$. From the results of the above two specific problems, we can
  see that the upper bound $\Delta (m,r_k)$ is smaller in general  for
  SPD matrices than for indefinite matrices.

\subsubsection{Spectral analysis of the preconditioned Schur complement}\label{spectral}
The preconditioning effect of the proposed PSLR preconditioner depends directly on the eigenvalue distribution of $S_{\text{app}}^{-1}S$. When the eigenvalues of $S_{\text{app}}^{-1}S$ are clustered or close to one, one can expect a fast convergence for Krylov subspace methods.

From (\ref{s1}) and (\ref{Sinv0}), we have
\begin{eqnarray}\label{SinvS}\tiny
 S_{\text{app}}^{-1}S
 &=&   \bigg[\sum_{i=0}^{m}(C_{0}^{-1}E_{s})^{i}C_{0}^{-1}\bigg](I-V_{r_k}H_{r_k}V_{r_k}^{T})^{-1}
 (I-E_{rr}(m))\bigg[\sum_{i=0}^{m}(C_{0}^{-1}E_{s})^{i}C_{0}^{-1}\bigg]^{-1}\nonumber\\
   &=& \bigg[\sum_{i=0}^{m}(C_{0}^{-1}E_{s})^{i}C_{0}^{-1}\bigg]Z(r_k)^{-1}\big(Z(r_k)-X(m,r_k)\big)
   \bigg[\sum_{i=0}^{m}(C_{0}^{-1}E_{s})^{i}C_{0}^{-1}\bigg]^{-1},
\end{eqnarray}
where $Z(k)$ and $X(m,r_k)$ are the matrices defined by (\ref{denote}).
Obviously, it follows from (\ref{SinvS}) that $S_{\text{app}}^{-1}S$ is similar to
$$Z(r_k)^{-1}\big(Z(r_k)-X(m,r_k)\big)=I-Z(r_k)^{-1}X(m,r_k), $$
which implies that $$\lambda(S_{\text{app}}^{-1}S)=1-\lambda(Z(r_k)^{-1}X(m,r_k)).$$

When the eigenvalues of $X(m,r_k)$ (or $Z(r_k)^{-1}X(m,r_k)$) are close to
  zero, the eigenvalues of $S_{\text{app}}^{-1}S$ are clustered around
  $1$. To illustrate the influence of the approximation accuracy of
  $S_{\text{app}}^{-1}$ on the eigenvalue distribution of
  $S_{\text{app}}^{-1}S$, we display the eigenvalues of
  $S_{\text{app}}^{-1}S$ for the same 3D Laplacian matrix presented in
  Section \ref{approximation} with $n=2000$, $r_k=15$ and $s=5$ in
  Figure \ref{eigenvalue}. The numbers of terms used in the power
  series expansion are $m=3$ and $m=5$ for two different cases,
  respectively.  As can be seen from Figure \ref{eigenvalue}, the
  eigenvalues in the right subfigure are more clustered than those in
  the left subfigure. This further illustrates the fact that the approximation
  is improved if $m$ increases but the rank $r_k$ is fixed. For this
  specific problem, the PSLR preconditioned GMRES method converges in
  $8$ and $17$ iterations, respectively, when $m$ is  set to $5$ and
  $3$ and iteration is stopped when the initial residual
  is reduced by $10^{8}$.
  As another illustration, Figure \ref{eigenvalues} depitcs the
  eigenvalues of $S_{\text{app}}^{-1}S$ for the same shifted 3D
  Laplacian matrix presented in Section \ref{approximation} with
  $n=2000$, $r_k=15$ and $s=5$. For this test, the PSLR preconditioned
  GMRES method with $m = 5$ converges in $13$ iterations and the
  iteration number increases to $24$ when $m$ is reduced to $3$, using the
  same stopping criterion as earlier.

\begin{figure}[ptbh]
\centering\tabcolsep=0mm
	\begin{tabular}
		[r]{cc}%
		\includegraphics[height=1.80in]{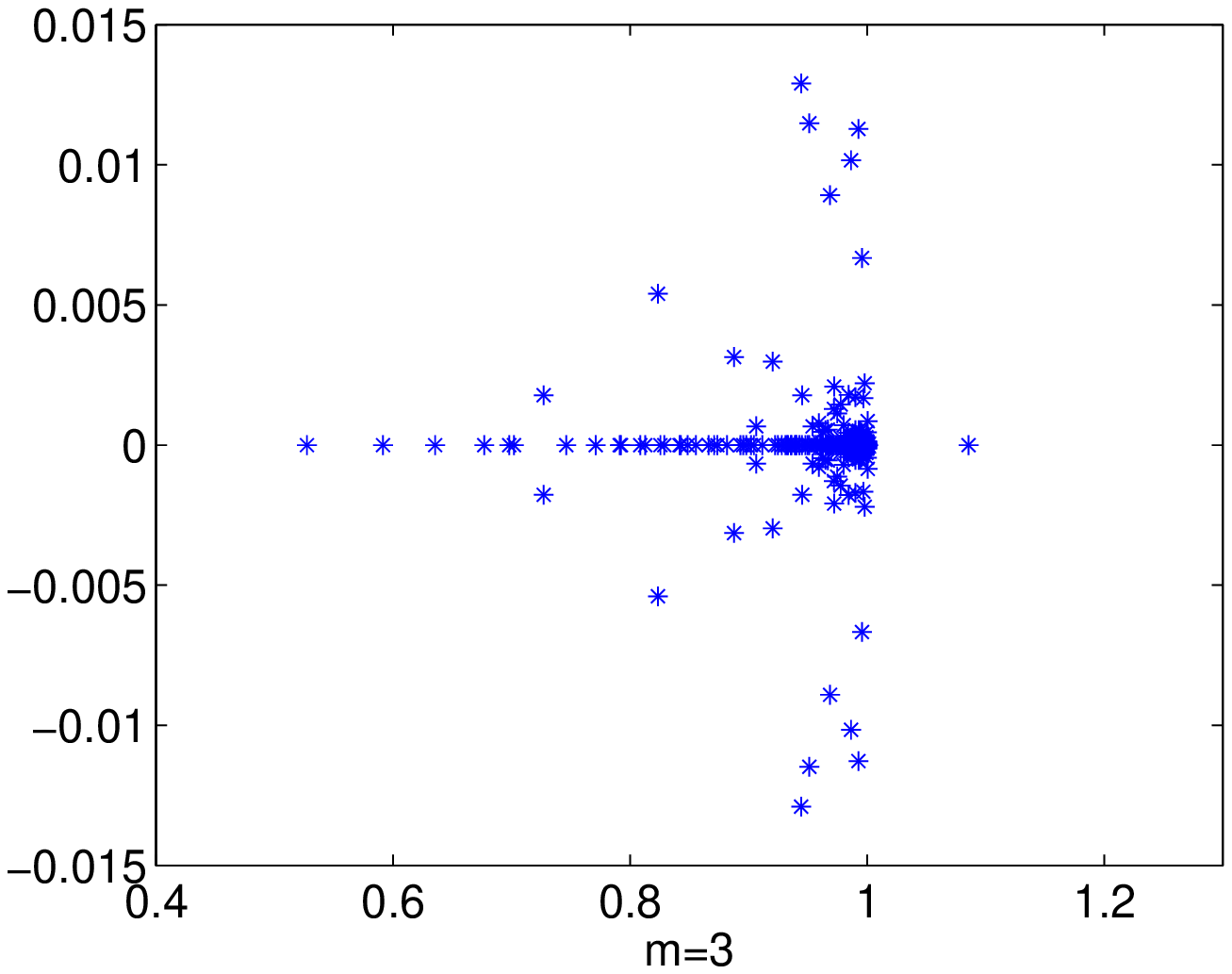}&
		\includegraphics[height=1.80in]{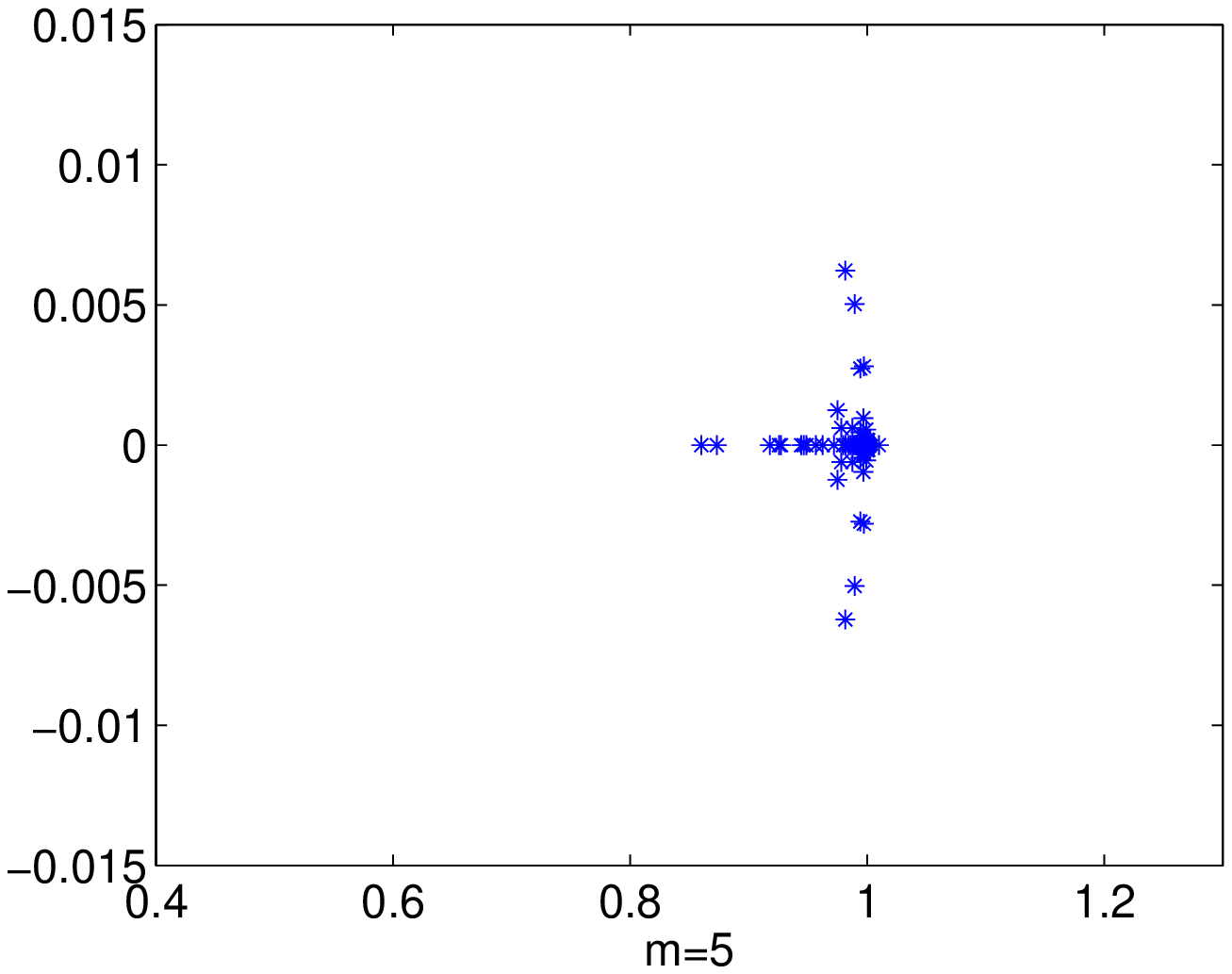}
	\end{tabular}
	\caption{The eigenvalue distribution of $S_{\text{app}}^{-1}S$ for 3D Laplacian matrix with $r_k=15$.
The number of terms used in the power series expansion are $3$ and $5$ for
left subfigure and right subfigure, respectively.}%
	\label{eigenvalue}%
\end{figure}

\begin{figure}[ptbh]
\centering\tabcolsep=0mm
	\begin{tabular}
		[r]{cc}%
		\includegraphics[height=1.80in]{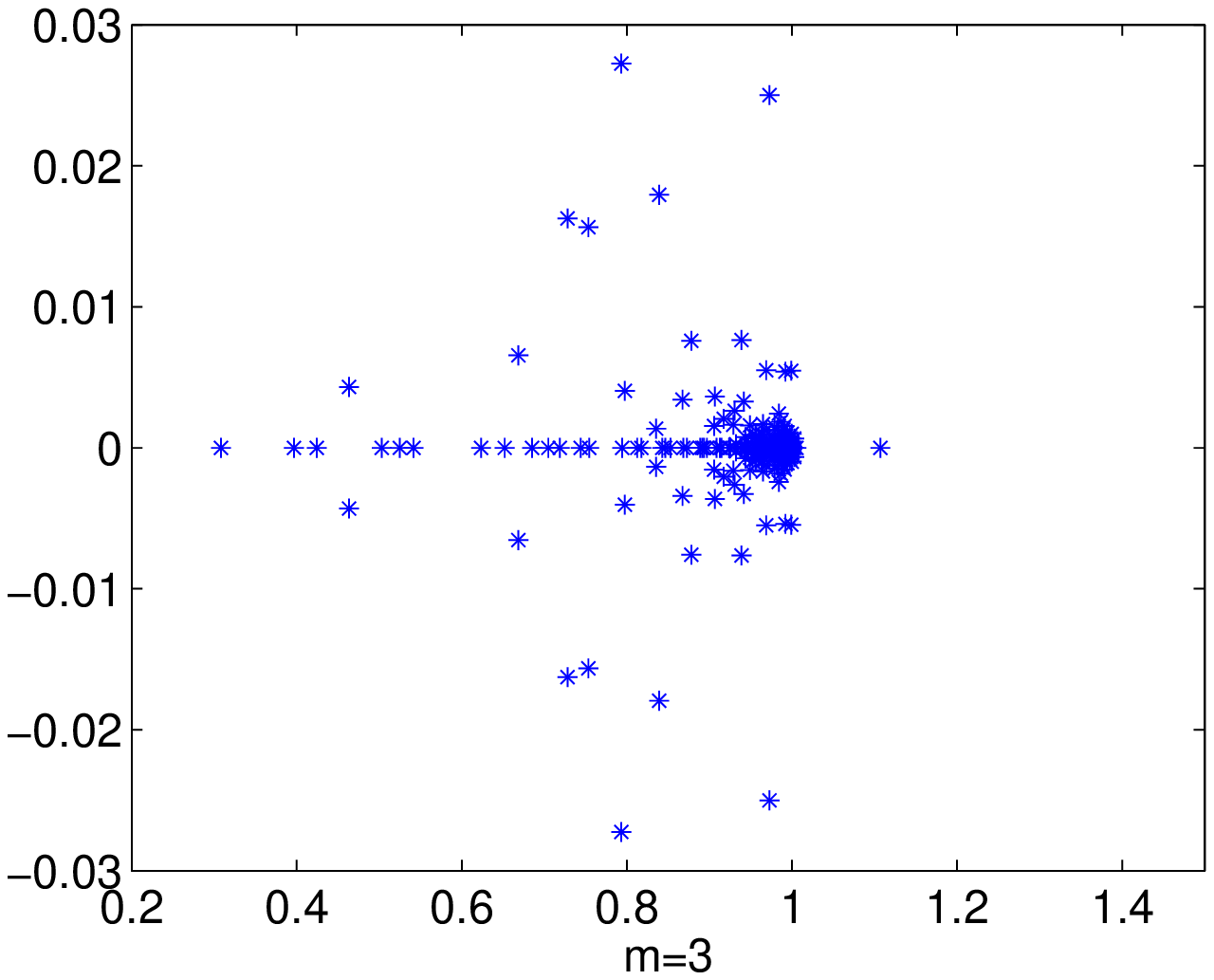}&
		\includegraphics[height=1.80in]{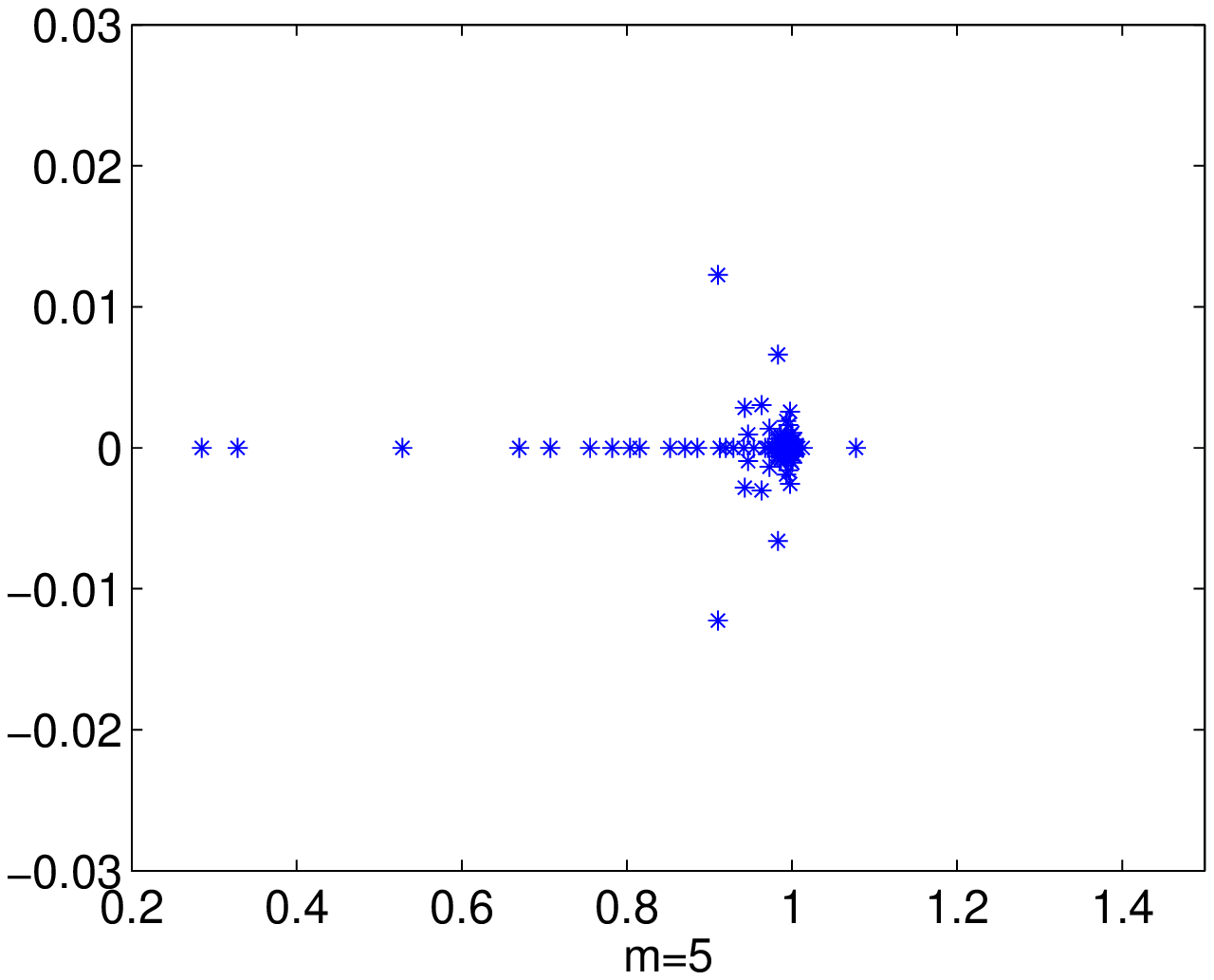}
	\end{tabular}
	\caption{The eigenvalue distribution of $S_{\text{app}}^{-1}S$ for shifted 3D Laplacian matrix with $r_k=15$.
The number of terms used in the power series expansion are $3$ and $5$ for
left subfigure and right subfigure, respectively.}%
	\label{eigenvalues}%
\end{figure}

\subsection{Construction and application of the PSLR preconditioner}
This section provides a short description of the construction of
the PSLR preconditioner and its
application. Recall from \eqref{equ:solve},
the application of the PSLR preconditioner on a vector $b$ follows the
following two steps:
\begin{equation}\label{eq:application}
  \left\{\begin{array}{lll}
  y= S_{\text{app}}^{-1}(g-FB^{-1}f),  \\[2mm]
  x=B^{-1}(f-Ey),
\end{array}\right.
\end{equation}
where $b$ is partitioned into $(f^{T},g^{T})^{T}$
according to the sizes of $B$ and $C$.

The scheme \eqref{eq:application} requires three linear system solutions, two associated with $B$ and one associated with $S_{\text{app}}$. Applying $S_{\text{app}}^{-1}$ on a vector based on \eqref{Sinv} involves solving $m+1$ linear systems associated with $C_0$. Since both $B$ and $C_0$ are block diagonal, the construction of the PSLR preconditioner starts with the ILU factorization of these diagonal blocks. The computed ILU factors can then be used in the Arnoldi procedure to compute $V_k$ and $H_k$ associated with $S_{\text{app}}^{-1}$. The construction algorithm is summarized in Algorithm \ref{alg4}.

\begin{algorithm}
	\caption{Construction of PSLR preconditioner}\label{alg4}
	\begin{algorithmic}[1]
	   \State Apply domain decomposition to reorder $A$ with $s$ subdomains
         \State {\bf For} $i=1:s$ Do
        \State ~~~~~$[L_{i}^{B},U_{i}^{B}]=\text{ilu}(B_{i})$
        \State ~~~~~$[L_{i}^{C_{0}},U_{i}^{C_{0}}]=\text{ilu}(C_{i})$
        \State {\bf EndDo}
		\State Apply Arnoldi procedure to compute:
		\State ~~~~~~~~~~$[V_{r_k},H_{r_k}]=\text{Arnoldi}(E_{rr}(m),r_k)$
		\State Compute $G_{r_k}=(I-H_{r_k})^{-1}-I$
	\end{algorithmic}
\end{algorithm}

The computational cost of the PSLR construction process is dominated
by ILU factorization and the triangular solves involved in applying
the operator $E_{rr}(m)$ in the Arnoldi process. Since these operations can
be performed independently among different diagonal blocks in $B$ and
$C_0$, Algorithm \ref{alg4} is highly parallelizable.

%

Algorithm \ref{alg3} describes the application of the PSLR
preconditioner on a vector $b$. Besides linear system solutions
associated with $B$ and $C_0$, the remaining operations are
matrix-vector multiplications associated with sparse matrices $E$,
$ F$ and dense matrices $V_{r_k}$ and $G_{r_k}$. Since both $E$ and $F$ are in
block diagonal forms, this application algorithm is also highly
parallizable.
\begin{algorithm}
	\caption{Computing $z=\text{PSLR}(b)$}\label{alg3}
	\begin{algorithmic}[1]
		\State Partition $b=\begin{pmatrix}
                                    f \\
                                   g \\
                                  \end{pmatrix}$
		\State Compute  $y=(g-FB^{-1})f$
		\State Update  $y\leftarrow y+V_{r_k}(G_{r_k}(V_{r_k}^{T}y))$
		\State Compute
{
$$y\leftarrow\sum_{i=0}^{m}(C_{0}^{-1}E_{s})^{i}C_{0}^{-1}y
$$
}
		\State Solve $Bx=f-Ey$
        \State Set $z=\begin{pmatrix}
                                    x \\
                                    y \\
                                  \end{pmatrix}$
	\end{algorithmic}
\end{algorithm}

\section{Numerical examples}\label{sec:num}

In this section, we report numerical experiments to show the
efficiency and robustness of the PSLR preconditioner. The test
problems include symmetric and nonsymmetric cases.  The PSLR
preconditioner was implemented in  C++ and compiled with
the  -O3 optimization option.  All the experiments were run on a
single node of the \texttt{Mesabi} Linux cluster at the Minnesota
Supercomputing Institute, which has 64 GB or memory and two Intel
2.5 GHz Haswell processors with 12 cores each.  The \texttt{PartGraphKway}
from the METIS \cite{Karypis} package was used to partition matrices.  BLAS and LAPACK
routines from Intel Math Kernel Library (MKL) were used to enhance the performance on multiple cores. Thread-level parallelism was realized by OpenMP. The preconditioner construction time consists of the
incomplete LU factorizations of matrices $B_{i},C_{i}$, $i=1,2,\ldots,s$,
and the computation of $V_{r_k}$ and $G_{r_k}$.
In actual computations, the right-hand side $b$ was chosen randomly such
that $Ax=b$ with $x$ being a random vector, and the initial guess
on $x$ was always taken as a zero vector in the Krylov subspace methods.

For the SPD problems, we compare the PSLR preconditioner with the MSLR
preconditioner \cite{m8} and the incomplete Cholesky factorization
preconditioner (ICT) with threshold dropping, and the conjugate
gradient (CG) method as the accelerator.  For general problems, we
compare PSLR  with the GMSLR preconditioner and the
incomplete LU factorization preconditioner (ILUT) with threshold
dropping, using   GMRES \cite{Saad-book2} as the accelerator. BLAS and LAPACK
routines from Intel Math Kernel Library (MKL) were used in incomplete factorizations and MSLR and GMSLR precoditioners. MSLR and GMSLR preconditioners were also parallelized with OpenMP.

In the rest of this section, the following notation is used:
\begin{itemize}
  \item its: the number of iterations of GMRES or CG to reduce the initial residual norm
   by $10^{8}$. Moreover, the "F" indicates that GMRES or CG failed to
   converge within 500 iterations;
    \item    o-t: wall clock time to reorder the matrix;
  \item p-t: wall clock time for the preconditioner construction;

  \item i-t: wall clock time for the iteration procedure. If GMRES or CG
  fails to converge within 500 iterations, then we denote this time by "--";

  \item t-t: total wall clock time, i.e., the sum of the preconditioner construction time and the
  iteration time;
  \item $r_{k}$: the rank used in the low-rank correction terms;

  \item $m$: the number of terms used in the power series expansion;

  \item fill (total): the total fill-factor defined as $\frac{\text{nnz}(prec)}{\text{nnz}(A)}$;

  \item fill (ILU): the fill-factor comes from ILU decompositions
  defined as $\frac{\text{nnz}(ILU)}{\text{nnz}(A)}$;

  \item fill (Low-rank): fill-factor comes from the low-rank correction terms
  defined as $\frac{\text{nnz}(LRC)}{\text{nnz}(A)}$.
\end{itemize}
Here $\text{nnz}(X)$ denotes the number of nonzero entries of
a matrix $X$. Moreover,
\begin{eqnarray*}
  && \text{nnz}(ILU)=\sum_{i=1}^{s}\big[\text{nnz}(L_{i}^{B})+\text{nnz}(U_{i}^{B})+
  \text{nnz}(L_{i}^{C_{0}})+\text{nnz}(U_{i}^{C_{0}})\big],\\
  && \text{nnz}(LRC)=\text{nnz}(V_{r_k})+\text{nnz}(G_{r_k}), \\
  && \text{nnz}(prec)=\text{nnz}(ILU)+\text{nnz}(LRC).
\end{eqnarray*}
Note that we employ the notation $\text{nnz}(V_{r_k})$
and $\text{nnz}(G_{r_k})$ for dense matrices. The  term
\emph{fill-factor}, which is meant to reflect memory usage,
mixes traditional fill-in (ILU) along with the  additional
memory needed to store the (dense) low-rank correction matrices.

\subsection{Test 1}
Consider the following symmetric problem:
\begin{eqnarray}\label{test1}
-\triangle u-\beta u &=& f ~\text{in}~\Omega,\\ \nonumber
u &=& 0  ~\text{on}~\partial\Omega.
\end{eqnarray}
Here $\Omega=[0,1]^{3}$ and these PDEs were discretized by the 7-point
stencil.  The discretized operation is equivalent to shifting the
discretized Laplacian by a shift of $h^{2}\beta I$ for a mesh spacing of $h$.

\subsubsection{Effect of $s$}\label{subdomain}

In this subsection, we look into the effect of the number $s$ of
subdomains on the effectiveness of the PSLR preconditioner.  We solve
(\ref{test1}) with the shift of $0.05$ on a 50$^{3}$ grid by the
GMRES-PSLR method.  The resulting coefficient matrix is
indefinite. The number of terms used in the power series expansion is
$m=3$, and the rank for the low-rank correction terms is fixed at $15$.

\begin{table}[h]
	\caption{The fill-factor, iteration counts and CPU time for
solving (\ref{test1}) with $\text{shift}=0.05$ on a 50$^{3}$ grid by
the GMRES-PSLR method ($m=3$). Here, the rank in  the low-rank
correction terms is $15$, and the dropping threshold in the
incomplete LU factorizations is $10^{-2}$.
		\label{tab:nd}}
	\begin{center}
		\small
		\begin{tabular}{r|c|c|c|c|c|c}
			\hline
			s & fill (ILU)& fill (Low-rank)& fill (total) & its &p-t&i-t
			\tabularnewline
			\hline
                                                                                  5 &    2.68   &   .19&    2.88  &  90 & .11 & 1.10 \tabularnewline
                                                                                 15&    2.45   &   .37  & 2.83    &  89 & .13 & .64 \tabularnewline
			25&    2.30   &  .49 &    2.79 &   86 & .15 & .56\tabularnewline
			35 &    2.23   &   .55 &   2.78  &  83 &  .16& .54 \tabularnewline
			45 &    2.17   &   .60  &  2.77 &    80 & .19 &.65 \tabularnewline
			55 &    2.11   &   .66 &  2.77  & 78& .20&.67\tabularnewline
			\hline
		\end{tabular}
	\end{center}
\end{table}
We can see from Table \ref{tab:nd} that the
fill-factor from ILU decompositions decreases monotonically
while the fill-factor
from low-rank correction terms increases when $s$ increases
from $5$ to $55$. This is because the size of
each $B_{i}$ and $C_{i}$, $i=1,2,\ldots,s$ is smaller from a larger $s$,
which reduces the storage and the computational cost for the ILU
factorizations. A larger $s$ also results in a larger Schur complement $S$, which implies that the matrix $V_{r_k}$
has more rows. That is why the fill-factor
from the low-rank correction terms increases when $s$ becomes larger. These experiments also illustrate the fact that
the performance of the PSLR preconditioner does not vary much with the number of subdomains used.

\subsubsection{Effect of $m$}\label{m}

The number of terms used in the power series expansion is also an
important factor, as was previously discussed.  We investigate this factor
by solving the same problem as in Section \ref{subdomain} with the rank
used in the low-rank correction part being fixed at 15.  The iteration
counts and CPU times for different $m$'s are given in Table
\ref{tab:steps}. As can be observed, the iteration number
decreases from 171 to 78 when $m$ increases from $0$ to $5$. This can be
attributed to the improved clustering of the spectrum of the preconditioned Schur complement
as  the number of terms used in the power series
expansion increases.  Meanwhile, the time to construct the PSLR
preconditioner increases slightly.  Since the iteration number is
reduced considerably when $m$ increases from $0$ to some positive
constant and then reduced slightly after that, we expect that the
iteration time decreases first and then increases. This is verified
by the numerical results in Table \ref{tab:steps}. As is seen from
Table \ref{tab:steps}, the iteration time  first goes down from $.90$
to $.57$ as $m$ increases from $0$ to $3$ and then increases from
$.57$ to $.63$ when $m$ increases from $3$ to $5$. The total time has
the same trend as that of the iteration time. The
results in Table \ref{tab:steps} are plotted in  Figure \ref{fig:m}. In the
figure we can see  that $m=3$ is optimal  for this
test, in terms of  CPU time.  In general, there is a similar pattern and
$m$ should not be taken too large for the sake of a better  overall
performance.

\begin{table}[h]
 	\caption{Iteration counts and CPU times for solving (\ref{test1})
 with $\text{shift}=0.05$ on a $50^{3}$ grid by the GMRES-PSLR method, in which
 $s=35$, the dropping threshold in the incomplete LU
 factorizations is $10^{-2}$, and the rank in
 the low-rank correction part is 15.
	\label{tab:steps}}
	\begin{center}
		\tabcolsep5mm
		\def\arraystretch{1.0}
		\small
		\begin{tabular}{c|c|c|c|c}
			\hline
			m& its&p-t& i-t& t-t
			\tabularnewline
			\hline
			0 &  171 &  .11   & .90 &  1.01\tabularnewline
			1 &  109 &  .12   & .61  &  .73\tabularnewline
			2 &  96 & .13   & .59 &.72\tabularnewline
			3 & 86  &  .14&  .57&.71 \tabularnewline
			4 & 81 & .16   & .60 &.76\tabularnewline
			5 & 78  &  .18 &  .63&.81 \tabularnewline
			\hline
		\end{tabular}
	\end{center}
\end{table}

\begin{figure}[ptbh]
	\begin{tabular}
		[r]{c}%
		\includegraphics[height=2.10in]{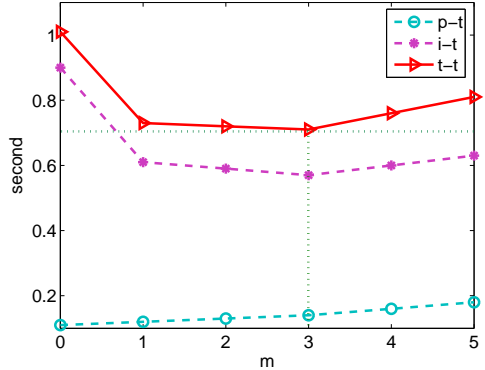}
	\end{tabular}
	\caption{The preconditioner construction time, the iteration time,
and the total time for solving (\ref{test1}) with $\text{shift}=0.05$
on a $50^{3}$ grid with different $m$'s by the GMRES-PSLR method.}%
	\label{fig:m}
\end{figure}

\subsubsection{Effect of $r_{k}$}\label{sec:rank}

In this subsection, we consider the effect of the rank used in
the low-rank correction terms on the PSLR preconditioner. Here
we consider the same test problem used in the previous two subsections but with different $r_k$'s. We observe from Table \ref{tab:rank}
that the iteration number decreases as $r_{k}$ increases from $0$ to $75$. The fill-factor
from ILU decompositions keeps the same value $2.44$ since we fix the number
of subdomains. On the other hand, the fill-factor from the low-rank terms and
the time to compute low-rank correction terms increase as the rank becomes larger.
In the meantime, the iteration time and the total CPU time decrease as $r_{k}$
increases from $0$ to $15$ and then increase. This indicates that there is no
need to take a very large rank in practice

\begin{table}[h]
	\caption{The fill-factor, iteration counts and CPU time for solving (\ref{test1})
with $s=0.05$ on a $50^{3}$ grid by the GMRES-PSLR method with $m=3$, in which $s=35$
and the dropping threshold in the incomplete LU factorizations is $10^{-2}$.
	\label{tab:rank}}
	\begin{center}
		\tabcolsep5mm
		\def\arraystretch{1.0}
		\small
		\begin{tabular}{c|c|c|c|c|c|c}
			\hline
			$r_{k}$&fill (ILU) &fill (Low-rank)&its&p-t& i-t& t-t
			\tabularnewline
			\hline
			0 & 2.24 &.00&92 &  .08   & .81 &  .89\tabularnewline
			15 &  2.24&.55&86 &  .13   & .57  &  .70\tabularnewline
			30 & 2.24 &1.10&83 & .19   & .59 &.78\tabularnewline
			45 &2.24 &1.65&80  &  .31&  .61&.92\tabularnewline
			60 &2.24 &2.20&78 & .33   & .60 &.93\tabularnewline
			75 &2.24&2.75& 75  &  .39 &  .60&.99 \tabularnewline
			\hline
		\end{tabular}
	\end{center}
\end{table}

\begin{table}[h]
	\caption{{The fill-factor, iteration counts and CPU time for solving (\ref{test1})
with $s=0.14$ on a $50^{3}$ grid by the GMRES-PSLR method with $m=3$, in which $s=35$
and the dropping threshold in the incomplete LU factorizations is $10^{-2}$.}
	\label{tab:rank-rk}}
	\begin{center}
		\tabcolsep5mm
		\def\arraystretch{1.0}
		\small
		\begin{tabular}{c|c|c|c|c}
			\hline
			$r_{k}$&fill (ILU) &fill (Low-rank)&its& t-t
			\tabularnewline
			\hline
			0 &3.07 &.00& F   & -- \tabularnewline
			15 &3.07 &.55& 346   &  8.90 \tabularnewline
            30 &3.07 &1.10& 310   &  8.06 \tabularnewline
            45 &3.07 &1.65& 266   &  7.01 \tabularnewline
            60 &3.07 &2.20& 220   &  5.65 \tabularnewline
            75 &3.07 &2.75& 199   &  5.69 \tabularnewline
			\hline
		\end{tabular}
	\end{center}
\end{table}

In addition, Table \ref{tab:rank-rk} shows the benefit of
incorporating low-rank corrections in the PSLR preconditioner when
solving highly indefinite linear systems. Note that the PSLR
preconditioner reduces to the Neumann polynomial preconditioner when
the rank $r_{k}$ is equal to zero. Results in Table \ref{tab:rank-rk}
show that the low-rank correction technique can greatly improve the
performance and robustness of the classical Neumann polynomial
preconditioner even when the rank $r_k$ is smaller than the number of the eigenvalues of $E_{rr}$ with modulus greater than $1$. For example, the GMRES-PSLR combination (full GMRES is
used) fails to converge when there is no low-rank correction
applied. Here, the shift for the grid $50^{3}$ is set to  $.14$, in
which case the shifted discretized operator has $78$ negative
eigenvalues.

\subsubsection{Effect of the number of threads}\label{sec:threads}
We now examine the effect of the  number of threads on the performance, {when parallelization is achieved through openMP.}
Table \ref{tab:nt} shows the total execution time as the number of
threads increases from 4 to 24, when solving Problem
  (\ref{test1}) with $s=0.05$ on a $50^{3}$ grid. The rank here is
  taken as $r_{k}=15$. As one can see from Table \ref{tab:nt}, the
  total wall clock time decreases as the number of threads increases.
   For this case, the total fill factor is $2.79$ ($2.24$ for ILU and
   $0.55$ for the low-rank part) and the iteration number is 86
   (regardless of the number of threads).
 As expected, the execution time for GMRES-PSLR is reduced
 when more threads are used, due to parallelism.
 So, the number of threads used in our
 numerical experiments is taken as the number of cores, i.e., $24$. {Note that the nodes used for the experiment have 12 cores, but due to hyperthreading
 up to 24 threads can be efficiently executed in parallel as is shown by the
 experiment.
}

\begin{table}[h]
  \caption{
    Execution time as a function of the number of threads
    for solving (\ref{test1})
    with $s=0.05$ on a $50^{3}$ grid by the GMRES-PSLR method with $m=3$,
    in which $s=35$, $r_k=15$
    and the dropping threshold in the incomplete LU factorizations is $10^{-2}$.
     \label{tab:nt}}
  \begin{center}
    \begin{tabular}{c|c}
      \hline
       Threads &  t-t \\ \hline
      4  &  5.02 \\
      8  &  2.26 \\
      16 &  1.28\\
      24 &  .70 \\ \hline
    \end{tabular}
  \end{center}
\end{table}

\subsubsection{Laplacian matrices}\label{Laplacian}

We now test some general 3D Laplacian matrices
to show the efficiency of the PSLR preconditioner.
We solve (\ref{test1}) with $\beta>0$,
where the corresponding problems are symmetric indefinite. For these
problems, the discretized Laplacian was shifted
by $h^{2}\beta I$ for mesh
size $h$. {The numbers of negative eigenvalues are $20,69,133$ for grids $32^{3}, 64^{3}$ and $128^{3}$, respectively.} Here, we set $m=3$, $r_k=15$ and $s=35$ in the PSLR preconditioner. As we see from Table \ref{tab:lap3d}, the PSLR
preconditioner outperforms ILUT and GMSLR preconditioners
for solving the resulting indefinite problems.
\begin{table}[h]
	\caption{Comparisons of {\rm PSLR} with $m=3$, $r_k=15$ and $m=35$, {\rm ILUT} and
		{\rm GMSLR} preconditioners for
		solving symmetric indefinite linear systems
from the 3-D shifted Laplacians \eqref{test1}.
		\label{tab:lap3d}}
	\begin{center}
		\def\arraystretch{1.5}
		\small
		\centering\tabcolsep2.1pt
		\begin{tabular}{c|c|ccccc|crcc|crccccc}
			\hline
			Mesh &  \text{shift} &
			\multicolumn{5}{c|}{PSLR} & \multicolumn{4}{c|}{ILUT} & \multicolumn{7}{c}{GMSLR}\tabularnewline
			&     & fill & its &o-t & p-t  & i-t  &fill &its  & p-t  & i-t &lev & $r_k$&fill &its &o-t& p-t  &i-t
			\tabularnewline
			\hline
			$32^{3}$  &  0.16  & 2.76  & 97  &.02 &.06 & .23 & 2.80 & 109 &.03  & .73 &7  &16 &2.75  & 106 &.03& .08  & .53\tabularnewline
			$64^{3}$  &$0.08$    &2.85  &288 &.15 &.26  &5.42& 2.86& 341& .29  & 25.72  & 10 & 16& 2.89 & 315& .22&.93& 18.57 \tabularnewline
			$128^{3}$  & 0.03   & 3.15 & 318  &.42&3.45 & 26.62 & 3.15 & F  & 2.16  & -- &13&16& 3.17  & F &.51& 5.32  & -- \tabularnewline
			\hline
		\end{tabular}
	\end{center}
\end{table}
This is because the iteration number and the construction time of the PSLR
preconditioner are much lower than those used by the other two preconditioners.
We found that the ILUT and GMSLR preconditioners cannot even converge
when the mesh size is $128^3$ and $\text{shift}=0.03$, in which case
the number of negative eigenvalues is $133$.

\subsection{Test 2}
We consider the shifted convection-diffusion equation below
\begin{eqnarray}\label{test2}
-\triangle u-\gamma\cdot\nabla u-\beta u &=& f ~\text{in}~\Omega,\\ \nonumber
u & =& 0  ~\text{on}~\partial\Omega,
\end{eqnarray}
which is a nonsymmetric problem. This equation is discretized by the standard 7-point stencil
in 3D, where $\Omega=[0,1]^{3}$ and $\gamma\in \mathbb{R}^{3}$.

Now we present more tests to illustrate the efficiency of PSLR
when  solving shifted convection-diffusion
equations. Here, $\gamma$ is set to $(0.1,0.1,0.1)$ and the
$\text{shift}$ is taken as $0.16,0.08,0.03$ for grid
$32^{3}, 64^{3}, 128^{3}$, respectively. Here, we fixed $m=3$, $r_k=15$ and $m=35$ in the PSLR preconditioner. As is seen from Table
\ref{tab:conv}, the PSLR preconditioner outperforms GMSLR and ILUT
preconditioners. Again GMRES does not converge with the
GMSLR and ILUT preconditioners for the case when the $\text{shift}=0.03$
and the mesh size is $128^3$.

\begin{table}[h]
	\caption{Comparisons of {\rm PSLR} with $m=3$, $r_k=15$ and $m=35$, {\rm ILUT} and
		{\rm GMSLR} preconditioners for nonsymmetric indefinite linear systems from the discretized 3-D shifted convection-diffusion equation \eqref{test2}.
		\label{tab:conv}}
	\begin{center}
		\def\arraystretch{1.5}
		\small
		\centering\tabcolsep2.1pt
		\begin{tabular}{c|c|ccccc|crcc|crccccc}
			\hline
			Mesh & \text{shift} &
			\multicolumn{5}{c|}{PSLR} & \multicolumn{4}{c|}{ILUT} & \multicolumn{7}{c}{GMSLR}\tabularnewline
			&     & fill & its &o-t & p-t  & i-t &fill &its  & p-t  & i-t &lev &$r_k$ & fill &its &o-t& p-t  &i-t
			\tabularnewline
			\hline
			$32^{3}$  &  0.16  & 2.78  & 88  &.02& .05& .22 & 2.79 & 89 &.03  & .54 &7  &16 &2.73  & 86 &.03& .09  & .45\tabularnewline
			$64^{3}$  &$0.08$    &2.86  &260 & .15&.27  &5.54& 2.88& 270& .28  & 25.05  & 10 & 16& 2.89 & 266& .22&1.04& 16.73 \tabularnewline
			$128^{3}$  & 0.03   & 3.13& 309  &.41&3.77 & 24.98 & 3.10 & F  & 4.26  & -- &13&16& 3.12  & F &.52& 3.56  & -- \tabularnewline
			\hline
		\end{tabular}
	\end{center}
\end{table}

\subsection{Test 3}
Next we test PSLR for some general sparse linear systems
including symmetric and nonsymmetric ones to show that the method can
work quite well for general systems. The test
matrices are from SuiteSparse Matrix Collection \cite{Davis} and   Table
\ref{tab:matrices} provides a brief description.
\begin{table}[h]
	\small
	\caption{Some details on the test matrices.}
	\begin{center}
		\tabcolsep1mm
		\def\arraystretch{1.35}
		\begin{tabular}{l|r|r|c|r}
			\hline
			\noalign{\smallskip}
			\multicolumn{1}{c|}{Matrix} & \multicolumn{1}{c|}{Order} & \multicolumn{1}{c|}{nnz}
			&symmetric & \multicolumn{1}{c}{Description}\tabularnewline
			\noalign{\smallskip}
			\hline
			\noalign{\smallskip}
			cfd1    & 70,656   & 1,825,580&yes  &CFD problem \tabularnewline
			ecology1 & 1,000,000 & 4,996,000 & yes &landscape ecology problem
			\tabularnewline
			ecology2 & 999,999 & 4,995,991 &  yes&landscape ecology problem
			\tabularnewline
			thermal1  & 82,654 & 574,458 & yes &thermal problem \tabularnewline
			thermal2  & 1,228,045 & 8,580,313 & yes &thermal problem \tabularnewline
			Dubcova3    & 146,689   & 3,636,643 &yes  &2D/3D problem \tabularnewline
			\hline
			CoupCons3D    & 416,800   &17,277,420 & no &structural problem \tabularnewline
			Atmosmodd  & 1,270,432  & 8,814,880   & no & atmospheric model \tabularnewline
			Atmosmodl    & 1,489,752  & 10,319,760 & no &atmospheric model\tabularnewline
			Cage14 & 1,505,785 & 27,130,349 & no &directed weighted graph
			\tabularnewline
			Transport& 1,602,111 & 23,500,731 & no&structural problem \tabularnewline
			\hline
		\end{tabular}
	\end{center}
	\label{tab:matrices}
\end{table}

Numerical results are presented in Table \ref{tab:general}. {Here, we fixed $s=35$, $m=3$ and $r_k=50$ in the PSLR preconditioner for all the experiments.}
From this table, we can see that the GMRES-PSLR method converges
for all the test problems without tuning its parameters. Moreover, the iteration time
is much less than that of MSLR, ICT, GMSLR and ILUT preconditioners.
The GMRES accelerator failed to converge within 500 iterations
when used in conjunction with the
ICT and MSLR preconditioners for the CFD problem \texttt{cfd1}.

\begin{table}[h]
	\caption{Comparisons of {\rm PSLR} with $m=3$, $r_k=15$ and $m=35$, {\rm ICT/ILUT}
and {\rm MSLR/GMSLR} preconditioners.
\label{tab:general}}
	\def\arraystretch{1.45}
	\small
	\begin{center}
		\centering\tabcolsep1.93pt
		\begin{tabular}{l|lcccc|lccc|cllcccc}
			\hline \noalign{\vspace{0.5pt}}
			\multicolumn{1}{c|}{Matrix}   & \multicolumn{5}{c|}{PSLR} &
			\multicolumn{4}{c|}{ICT}&
			\multicolumn{7}{c}{MSLR}\tabularnewline
			&   fill  & its &o-t & p-t  & i-t   & fill  & its  & p-t  &
			i-t &lev& $r_k$ &fill  & its  &o-t &p-t  &
			i-t\tabularnewline
			\hline
			\scriptsize{cfd1}  & 3.15 & 245&.12 &.51& 3.92 &  3.14 & F& 11.48 & -- & 7 &180 &3.15&F &.20& 10.9& --
			\tabularnewline
			\scriptsize{ecology1}   & 2.68 & 119 &.25& 1.24& 8.33 &  2.67 & 87 & .60 & 17.40 & 7 &32 & 2.68 & 318 &.36& 11.6& 9.97
			\tabularnewline
			\scriptsize{ecology2}   & 2.68 & 107 &.25 &1.25 & 9.57 &  2.67 & 402& .61 & 26.87 & 8 &35 & 2.67 &399 & .35&10.4& 15.3
			\tabularnewline
			\scriptsize{thermal1}   & 2.38& 103 &.15 &.15 & .38 & 2.38& 138 & .16 &  2.84&6 &24 &2.39 & 181 &.20& 1.09& .68
			\tabularnewline
			\scriptsize{thermal2}   & 2.44 & 156&.30& 2.06 & 12.17& 2.45& 317& 2.56 & 26.66  &8 &32 &2.46 &497 &.38& 20.9& 22.8
			\tabularnewline
			\scriptsize{Dubcova3}  & 3.62 & 61&.17 &.87 & 1.67 &  3.59 & 52 & 1.98 & 2.50& 8 &64 & 3.60 &23& .24& 1.58& 2.21
			\tabularnewline
			\hline \noalign{\vspace{0.5pt}}
			\multicolumn{1}{c|}{Matrix}   & \multicolumn{5}{c|}{PSLR} &
			\multicolumn{4}{c|}{ILUT}&
			\multicolumn{7}{c}{GMSLR}\tabularnewline
			&   fill  & its &o-t & p-t  & i-t   & fill  & its & p-t  &
			i-t &lev& $r_k$ &fill  & its &o-t  & p-t  &
			i-t\tabularnewline
			\scriptsize{CoupCons3D}  & 1.54 & 19& .20&1.76& 2.40  & 1.53 & 12 &  8.64 &4.21 &10 &16 & 1.53& 17&.29&2.35 & 3.51
			\tabularnewline
			\scriptsize{Atmosmodd}  & 4.24 & 36&.35&2.40& 6.88  & 4.28 & 45&12.73 & 17.11 &10 &16 & 4.26 & 38&.47&4.0 & 15.78
			\tabularnewline
			\scriptsize{Atmosmodl}   & 4.65& 18&.40 & 4.04 & 10.46&4.66 & 27 &  8.87 & 19.09&11 &16 & 4.62 & 25 &.51& 5.33& 16.22
			\tabularnewline
            \scriptsize{cage14}  & 2.13 & 4& .42&4.13 & 5.79  & 2.11 & 6& 6.95 & 10.18 &6&4 & 2.13 & 38&.51&5.73 & 8.89
			\tabularnewline
			\scriptsize{Transport} & 2.67& 99& .48&5.27 & 19.72 & 2.67& 100& 24.38  & 40.94 &11 &16 & 2.66 & 53 &.60 &6.09& 31.94
			\tabularnewline
			\hline
		\end{tabular}
	\end{center}	
\end{table}
\section{Conclusion}\label{sec:con}
We have presented an effective Schur complement-based
parallel preconditioner for solving general large sparse
linear systems. The method utilizes a standard Schur complement viewpoint
and exploits a power series expansion along with a low-rank correction
technique to approximate the inverse of the Schur complement.
The main difference between PSLR and  other Schur complement
techniques proposed earlier is that PSLR  relies on the
power series expansion to reduce the rank needed to obtain a good
approximation of the inverse of the Schur complement.  The number $m$
of terms used in the power series expansion and the rank used in the
low-rank correction part control the approximation accuracy of the
preconditioner.  In practice, small values for these two parameters
are  sufficient to yield  a reasonably good approximation to
$S^{-1}$.

As was illustrated in the experiments, a big advantage of PSLR is its
high level of parallelism.  Another advantage is its robustness when
solving indefinite linear systems.  Finally, PSLR is fairly easy to
build and apply and is quite general.  All that is required at the
outset is a problem that is partitioned into subdomains.  In our
future work, we will develop a general-purpose distributed memory
version of our current code.

\bibliographystyle{siam}
\bibliography{local2}
\end{document}